\documentclass{amsart}%
\usepackage{amsfonts}%
\usepackage{amsmath}%
\usepackage{amssymb}%
\usepackage{graphicx}
\providecommand{\U}[1]{\protect \rule{.1in}{.1in}}
\newtheorem{theorem}{Theorem}
\theoremstyle{plain}

\newtheorem{corollary}{Corollary}

\newtheorem{definition}{Definition}

\newtheorem{lemma}{Lemma}

\newtheorem{proposition}{Proposition}

\numberwithin{equation}{section}

\begin{document}
\title[Stability of group relations]{Stability of group relations under small Hilbert-Schmidt perturbations}
\author{Don Hadwin}
\address{University of New Hampshire}
\email{operatorguy@gmail.com}
\author{Tatiana Shulman}
\address{Institute of Mathematics of the Polish Academy of Sciences, Poland }
\email{tshulman@impan.pl}

\begin{abstract} If matrices almost satisfying a group relation are close to matrices exactly satisfying the relation, then we say that a group is
matricially stable. Here "almost" and "close" are in terms
of  the Hilbert-Schmidt norm. Using tracial 2-norm on $II_1$-factors we similarly define $II_1$-factor stability for groups.
Our main result is that all 1-relator groups with non-trivial center are $II_{1}$-factor stable.
Many of them are also matricially stable and RFD. For amenable groups we give a complete characterization of
matricial stability in terms of the following approximation property for characters: each character must be
a pointwise limit of traces of finite-dimensional representations. This allows us to prove matricial  stability
for the discrete Heisenberg group $\mathbb H_3$ and for all virtually abelian groups. For non-amenable groups the same approximation property is a necessary condition
for being matricially stable. We study this approximation property and show that RF groups with character rigidity have it.
\end{abstract}

\maketitle



\section*{Introduction}

Given an equation   of noncommutative variables  one can ask if
it is "stable", meaning that  each of its "almost" solutions  is "close" to a solution.

Examples of stability questions are famous questions about almost commuting matrices, which ask  whether almost commuting matrices are close to commuting ones.
The answers depend very much on classes of matrices and on the matrix norm one uses to measure "almost" and "close".
For instance for the operator norm those questions are due to Halmos (\cite{Halmos}). When matrices are two self-adjoint contractions the answer is positive by Lin's theorem (\cite{Lin}) and when they are two unitaries or three self-adjoint contractions, the answer is negative (\cite{Voiculescu}, \cite{Dav2}).
For the normalized Hilbert-Schmidt norm the question was formulated by Rosenthal \cite{Rosenthal} and has an affirmative answer for almost commuting unitaries, self-adjoint contractions and normal contractions (\cite{Glebsky}, \cite{FilonovKachkovskiy}, \cite{Hadwin-Li}).
In our recent work  \cite{MainPaper} we studied stability of not only commutator relations, but of general $C^*$-algebraic relations with respect to the normalized Hilbert-Schmidt norm and similar tracial norms on tracial $C^*$-algebras, in particular on $II_1$-factors. There we obtained far reaching generalizations of all the previous results (\cite{Glebsky}, \cite{FilonovKachkovskiy}, \cite{Hadwin-Li}).

The interest to stability questions with respect to the normalized Hilbert-Schmidt norm also has appeared recently in group theory, in the context of sofic and hyperlinear groups (\cite{Glebsky}, \cite{Glebsky2}, \cite{AP}). In particular one is interested in the question of whether permutation matrices almost satisfying a group relation are close to permutation matrices exactly satisfying the relation. Here "almost" and "close" are measured by the normalized Hamming distance.  For relations defining a finitely-generated abelian group it was answered  in the affirmative by Arzhantseva and Paunescu \cite{AP} (in fact they proved it not only with respect to the normalized Hammng distance but for arbitrary metrics). Although proving stability for permutations is not the same as for general unitary matrices and requires different techniques, however,  as was noticed in \cite{AP}, it has  similar flavor because the normalized Hamming distance can be expressed using the
Hilbert–-Schmidt
distance.

In this paper we focus on stability of group relations with respect to the normalized Hilbert-Schmidt norm and similar tracial norms.

Let $G$ be a finitely presented discrete group, and let
\[
G = \left< S|R\right>  =\left< g_{1},..., g_{s}\;|\;r_{1},...,r_{l}\right>
\]
be its presentation with $g_{i}$ being generators and $r_{j} = r_{j}
(g_{1},...,g_{s})$ being relations. We will say that $G$ is {\it matricially stable} if for any $\epsilon>0$ there is a $\delta>0$,  such that
if $k\in \mathbb N$ and  $U_1, \ldots, U_s$ are unitary $k\times k$ matrices satisfying

\[
\|1 - r_{j}\left( U_1, \ldots, U_s\right) \|_{2} \le \delta
\]
for all $j = 1, \ldots, l$, then there are unitary $k\times k$ matrices $U_1', \ldots, U_s'$ satisfying \[
r_{j}\left( U_1', \ldots, U_s'\right) =1
\]
for all $j = 1, \ldots, l$, and $\|U_i - U_i'\|_2 \le \epsilon$, for all $i=1, \ldots, s$.

This natural notion of stability can be easily generalized to arbitrary, not necessarily finitely presented,  discrete groups using tracial ultraproducts (see section 2 for the details). It implies in particular that the
property of being  matricially stable does not depend on the choice of a
generating set and a presentation.~ \footnote{In the context of sofic groups the fact that the stability of metric approximations (when a normalized bi-invariant metric is fixed on a class of approximating groups)
 does not depend on the choice of the generators and the presentation of a finitely presented groups is due to G. Arzhantseva and L. Paunescu \cite{AP}.}

 Using tracial 2-norm on $II_1$-factors we similarly define $II_1$-factor stability for groups (and some other versions of stability, see section 2).

For amenable groups we give a complete characterization of
matricial stability in terms of the following approximation property for characters: each character must be
a pointwise limit of traces of finite-dimensional representations (Theorem \ref{amenablegroup}). This allows us to prove matricial  stability
for the discrete Heisenberg group $\mathbb H_3$ (Theorem \ref{Heisenberg}) and for all virtually abelian groups (Theorem \ref{VirtAb}). For non-amenable groups the same approximation property is a necessary condition
for being matricially stable (Theorem \ref{NecessaryConditionGroups}). Thus it is very interesting for us to know what groups have this approximation property.
Recall that a group $G$ has character rigidity if the only extremal characters of $G$ which are not induced from the center
 are the traces of finite-dimensional representations (\cite{PetersonThom}).
In Corollary \ref{RFCharRig} we prove that RF groups with character rigidity have the approximation property above.

One of the main results of the paper is $II_1$-factor stability for a big class of non-amenable groups, namely for all 1-relator groups with non-trivial center (Theorem \ref{1-relatorTheorem}).
Many of those groups are also matricially stable (Theorem \ref{1-relator1}) and RFD (Theorem \ref{RFD}). By a group being RFD we mean that its full $C^*$-algebra is residually finite-dimensional (it is not the same as being residually finite (RF)).

\bigskip

\noindent \textbf{Acknowledgements.} The first author gratefully acknowledges a
Collaboration Grant from the Simons Foundation. The research of the
second-named author was supported by the Polish National Science Centre grant
under the contract number DEC- 2012/06/A/ST1/00256, by the grant H2020-MSCA-RISE-2015-691246-QUANTUM DYNAMICS and from the Eric Nordgren
Research Fellowship Fund at the University of New Hampshire.

\section{Preliminaries}

\noindent If a unital C*-algebra $\mathcal{B}$ has a tracial state $\rho$, we
define a seminorm  $\left \Vert \cdot \right \Vert
_{2}=\left \Vert \cdot \right \Vert _{2,\rho}$ on $\mathcal B$  by%
\[
\left \Vert b\right \Vert _{2}=\rho \left(  b^{\ast}b\right)  ^{1/2}.
\]

Suppose $I$ is an infinite set and $\alpha$ is an ultrafilter on $I$. We say
$\alpha$ is \emph{nontrivial} if there is a sequence $\left \{  E_{n}\right \}
$ in $\alpha$ such that $\cap_{n}E_{n}=\varnothing$. Suppose $\alpha$ is a
nontrivial ultrafilter on a set $I$ and, for each $i\in I$, suppose
$\mathcal{A}_{i}$ is a unital C*-algebra with a tracial state $\rho_{i}$.
By $%
{\displaystyle \prod_{i\in I}}
\mathcal{A}_{i}$ we will denote the $C^*$-product of the $C^*$-algebras $\mathcal A_i$, that is the $C^*$-algebra
$$ %
{\displaystyle \prod_{i\in I}}
\mathcal{A}_{i} = \{(a_i)_{i\in I}\;|\; a_i \in \mathcal A_i, \sup_{i\in I} \|a_i\| <\infty\}$$ with the norm given by $\|(a_i)_{i\in I}\| = \sup_{i\in I} \|a_i\|$. Note that sometimes one uses another notation for that,
$\oplus^{l^{\infty}}\mathcal A_i$, see \cite{Pestov}.

The
\emph{tracial ultraproduct} $%
{\displaystyle \prod_{i\in I}^{\alpha}}
\left(  \mathcal{A}_{i},\rho_{i}\right)  $ is the C*-product $%
{\displaystyle \prod_{i\in I}}
\mathcal{A}_{i}$ modulo the ideal $\mathcal{J}_{\alpha}$ of all elements
$\left \{  a_{i}\right \}  $ in $%
{\displaystyle \prod_{i\in I}}
\mathcal{A}_{i}$ for which%
\[
\lim_{i\rightarrow \alpha}\left \Vert a_{i}\right \Vert _{2,\rho_{i}}^{2}%
=\lim_{i\rightarrow \alpha}\rho_{i}\left(  a_{i}^{\ast}a_{i}\right)  =0.
\]
We denote the coset of an element $\left \{  a_{i}\right \}  \in %
{\displaystyle \prod_{i\in I}}
\mathcal{A}_{i}$ by $\left \{  a_{i}\right \}  _{\alpha}$.

Tracial ultraproducts for factor von Neumann algebras was first introduced by
S. Sakai \cite{Sakai} where he proved that a tracial ultraproduct of finite
factor von Neumann algebras is a finite factor. More recently, it was shown in
\cite{Hadwin-Li} that a tracial ultraproduct $%
{\displaystyle \prod_{i\in I}^{\alpha}}
\left(  \mathcal{A}_{i},\rho_{i}\right)  $ of C*-algebras is always a von
Neumann algebra with a faithful normal tracial state $\rho_{\alpha}$ defined
by%
\[
\rho_{\alpha}\left(  \left \{  a_{i}\right \}  _{\alpha}\right)  =\lim_{i\rightarrow
\alpha}\rho_{i}\left(  a_{i}\right)  .
\]
If there is no confusion, we will denote it just by $\rho$.

  The $C^*$-algebra of all complex n by n matrices will be denoted by $M_n(\mathbb C)$.
It has a unique tracial state  $tr_n = tr.$
By $\tau_{\alpha}$ we denote the corresponding tracial state on the tracial ultraproduct $\prod_{n\in \mathbb N}^{\alpha} (M_n(\mathbb C), tr_n).$

For a unital $C^*$-algebra $\mathcal A$, its unitary group will be denoted by $\mathcal U\left(\mathcal A\right)$.

\section{Stability for groups}

Let $G$ be a finitely presented discrete group, and let
\[
G = \left< S|R\right>  =\left< g_{1},..., g_{s}\;|\;r_{1},...,r_{l}\right>
\]
be its presentation with $g_{i}$ being generators and $r_{j} = r_{j}
(g_{1},...,g_{s})$ being relations. We assume that the set $S = \{g_{1}%
,...,g_{s}\}$ is symmetric, i.e. for every $g_{i}$ it contains $g^{-1}_{i}$ too.

Let $\mathcal{C}$ be a class of $C^{*}$-algebras and let $\mathcal{A }%
\in \mathcal{C}$ be unital with a tracial state $\rho$.

\begin{definition} $f: S
\to \mathcal{U}(\mathcal{A})$ is an \textit{$\epsilon$-almost homomorphism} if
\[
\|1 - r_{j}\left( f(g_{1}), \ldots, f(g_{s})\right) \|_{2, \rho} \le \epsilon
\]
for all $j = 1, \ldots, l$. \footnote{A similar notion was introduced in \cite{Man1,  Man3} with the difference that there the operator norm was involved.
  One should distinguish $\epsilon$-almost homomorphisms from completely different notions of group
 quasi-representations and $\delta$-homomorphisms as in \cite{Shtern, BOT}, where almost multiplicativity is required  on the whole  group.}
\end{definition}

\begin{definition} $G$ is \textit{$\mathcal{C}$-tracially stable} if for any
$\epsilon> 0$ there is $\delta>0$ such that for any  unital $C^{*}$-algebra
$\mathcal{A}\in \mathcal{C}$ with a tracial state $\rho$ and for any $\delta
$-almost homomorphism $f: S \to \mathcal{U}(\mathcal{A})$ there is a
homomorphism $\pi: G \to \mathcal{U}(\mathcal{A})$ such that
\[
\| \pi(g) - f(g)\|_{2, \rho} \le \epsilon
\]
for any $g\in S$.
\end{definition}

This natural notion of stability can easily be generalized for arbitrary
discrete, not necessarily finitely presented, groups.

\begin{definition}
$G$ is \textit{$\mathcal{C}$-tracially stable} if for any tracial ultraproduct
$%
{\displaystyle \prod_{i\in \mathbb N}^{\alpha}}
\left(  \mathcal{A}_{i},\rho_{i}\right)  $ of unital $C^{*}$-algebras
$\mathcal{A}_{i}\in \mathcal{C}$ with a trace $\rho_i$, any homomorphism $f: G \to \mathcal{U }\left(
{\displaystyle \prod_{i\in \mathbb N}^{\alpha}}
\left(  \mathcal{A}_{i},\rho_{i}\right)  \right) $ is liftable,
meaning that  for each $i\in \mathbb N$ there is a
homomorphism $f_{i}: G \to \mathcal{U}(\mathcal{A}_{i})$ such that $f(g) =
\{f_{i}(g)\}_{\alpha}.$
\end{definition}

\medskip

 We will show now that for a finitely presented group these two definitions of stability
coincide. It will imply in particular that in the first definition the
property of being  stable does not depend on the choice of a
generating set and a finite presentation.

\begin{proposition} For a finitely presented group the two definitions of stability above
coincide.
\end{proposition}
\begin{proof}
Let \[
G = \left< S|R\right>  =\left< g_{1},..., g_{s}\;|\;r_{1},...,r_{l}\right>.
\]
To see that the first definition of stability implies the second one, assume that $G$ is $\mathcal C$-stable with
respect to the first definition of stability and let
 $f:G\rightarrow  \mathcal{U }\left(%
{\displaystyle \prod_{i\in \mathbb N}^{\alpha}}
\left(  \mathcal{A}_{i},\rho_{i}\right)\right)  $ be a homomorphism.
First of all we notice that any unitary in $%
{\displaystyle \prod_{i\in \mathbb N}^{\alpha}}
\left(  \mathcal{A}_{i},\rho_{i}\right)$ can be lifted to a unitary in $\prod_{i\in \mathbb N} \mathcal A_i$.
Indeed let $u\in %
{\displaystyle \prod_{i\in \mathbb N}^{\alpha}}
\left(  \mathcal{A}_{i},\rho_{i}\right)$ be a unitary. Let $\mathbb T$ be the unit circle. Since $C(\mathbb T)$ is the universal $C^*$-algebra generated by one unitary,
there is a $\ast$-homomorphism $\phi: C(\mathbb T) \to %
{\displaystyle \prod_{i\in \mathbb N}^{\alpha}}
\left(  \mathcal{A}_{i},\rho_{i}\right)$ such that $\phi(z) = u$ (here $z\in C(\mathbb T)$ is the identity function).
By [\cite{MainPaper}, Th.5.3] applied to $\mathbb T$, $\phi$ can be lifted to a $\ast$-homomorphism $\psi:  C(\mathbb T) \to \prod_{i\in \mathbb N} \mathcal A_i$. Then $\psi(z)$ is a unitary lift
of $u$.

Thus
 for each $1\leq k\leq s$, we can write
\[
f \left(  g_{k}\right)  =\left \{  g_{k}\left(  i\right)  \right \}  _{\alpha
},
\] for some $g_k(i)\in \mathcal U \left(\mathcal A_i\right),$ $i\in \mathbb N$.
We then have, for each $j\le l$, %
\begin{multline*}
0=\left \Vert f \left(  r_j\left(  g_{1},\ldots g_{s}\right)  \right) - 1
\right \Vert _{2,\rho_{\alpha}}=\left \Vert r_j\left(  f \left(  g_{1}\right)
,\ldots f \left(  g_{s}\right)  \right) -1 \right \Vert _{2,\rho_{\alpha}}%
=\\ \lim_{i\rightarrow \alpha}\left \Vert r_j\left(  g_{1}\left(  i\right)
,\ldots,g_{s}\left(  i\right)  \right) -1 \right \Vert _{2,\rho_{i}}.
\end{multline*}
Since $\alpha$ is a nontrivial ultrafilter on $\mathbb N$, there is a decreasing
sequence\medskip \ $E_{1}\supset E_{2}\supset \cdots$ in $\alpha$ such that
$\cap_{k\in \mathbb{N}}E_{k}=\varnothing$. Since  $G$ is
$\mathcal C$-stable with respect to the first definition,  for each positive integer $m$ there is a number
$\delta_{m}>0$ such that, when $$\left \Vert r_j\left(  g_{1}\left(  i\right)
,\ldots,g_{s}\left(  i\right)  \right) -1 \right \Vert _{2,\rho_{i}}<\delta_{m}, $$ $j\le l$,
there is a homomorphism $\gamma_{m,i}:G%
\mathbf{\rightarrow}\mathcal U(\mathcal{A}_{i})$ such that%
\[
\max_{1\leq k\leq s}\left \Vert g_{k}\left(  i\right)  -\gamma_{m,i} \left(
g_{k}\right)  \right \Vert _{2,\rho_{i}}<1/m.
\]
Since $\lim_{i\rightarrow \alpha}\left \Vert r_j\left(  g_{1}\left(  i\right)
,\ldots,g_{s}\left(  i\right)  \right)-1  \right \Vert _{2,\rho_{i}}=0$ we can
find a decreasing sequence $\{A_{n}\}$ in $\alpha$ with $A_{n}\subset E_{n}$
such that, for every $i\in A_{n}$%
\[
\left \Vert r_j\left(  g_{1}\left(  i\right)  ,\ldots,g_{s}\left(  i\right)
\right) -1 \right \Vert _{2,\rho_{i}} \le \delta_{n}.
\]
For $i\in A_{n}\backslash A_{n+1}$ we define $f_{i}=\gamma_{n,i}$. For $i\in \mathbb N \backslash A_1$ we define $f_i$ arbitrarily. We then
have that $\left \{  f_{i}\right \}  _{i\in \mathbb  N}$ is a lifting
of $f$.

On the other hand, if $G$ is not $\mathcal C$-stable with respect to the first definition of stability, then there is
an $\varepsilon>0$ such that, for every positive integer $n$ there is a
 unital C*-algebra $  \mathcal{A}_{n}  $ with a trace $\rho_{n}$ and
$g_{1}\left(  n\right)  ,\ldots,g_{s}\left(  n\right) \in \mathcal U (\mathcal{A}_{n})  $ such that%
\[
\left \Vert r_j\left(  g_{1}\left(  n\right)  ,\ldots,g_{s}\left(  n\right)
\right)  -1\right \Vert _{2,\rho_{n}}<1/n,
\]
but for every homomorphism $\gamma: G%
\rightarrow \mathcal U(\mathcal{A}_{n})$%
\[
\max_{1\leq k\leq s}\left \Vert g_{k}\left(  n\right)  -\gamma \left(
g_{k}\right)  \right \Vert _{2,\rho_{n}}\geq \varepsilon.
\]
If we let $\alpha$ be any free ultrafilter on $\mathbb{N}$, we have that the
map $f$ defined by %
\[
f \left(  g_{k}\right)  =\left \{  g_{k}\left(  n\right)  \right \}  _{\alpha}%
\]
is a homomorphism from $G$ into $\mathcal U\left(%
{\displaystyle \prod_{n\in \mathbb{N}}^{\alpha}}
\left(  \mathcal{A}_{n},\rho_{n}\right)\right)  $ that is not liftable.

\end{proof}

Given a discrete group $G$ and a C*-algebra $A$, let $\pi:G \to \mathcal U(A)$ be a unitary representation of $G$ on $\mathcal U(A)$. Let $\mathbb C G$ denote the group algebra of $G$. Then $\pi$ induces a homomorphism $\pi:\mathbb C G \to A$. Recall that the {\it full $C^*$-algebra} $C^*(G)$ is the completion of $\mathbb C G$ with respect to the norm
$$\|a\|:=\sup\{\|\pi(a)\|: \pi:G\to \mathcal U(A) \; \text{is a homomorphism} \}.$$

The $C^*$-algebra $C^*(G)$ has the following universal property (which determines it uniquely up to
isomorphism of $C^*$-algebras). Given any $C^*$-algebra $A$ and any unitary representation $\pi : G \to \mathcal U(A)$, there exists a
unique $\ast$-homomorphism $\tilde \pi : C^*(G) \to A$ that satisfies $\tilde \pi(\delta(g)) = \pi(g)$ for every $g \in G$ (here $\delta: G \to \mathbb C G$ is the canonical embedding).

\medskip

In \cite{MainPaper} we introduced the following definition of $\mathcal C$-tracial stability for $C^*$-algebras. We call a $C^*$-algebra $A$ {\it $\mathcal C$-tracially stable} if for any ultrafilter $\alpha$ on $\mathbb N$ and any unital $C^{*}$-algebras
$\mathcal{A}_{i}\in \mathcal{C}$ with a trace $\rho_i$, any $\ast$-homomorphism $\phi: A \to 
{\displaystyle \prod_{i\in \mathbb N}^{\alpha}}
\left(  \mathcal{A}_{i},\rho_{i}\right)$ is liftable.

Our definition of stability for groups agrees with the definition of tracial stability for $C^*$-algebras in the following sense.

\begin{proposition}
A group $G$ is $\mathcal{C}$-stable iff its full $C^{*}$-algebra
$C^{*}(G)$ is $\mathcal{C}$-tracially stable.
\end{proposition}
\begin{proof}
Assume $G$ is $\mathcal{C}$-stable and let $\phi: C^*(G) \to %
{\displaystyle \prod_{i\in \mathbb N}^{\alpha}}
\left(  \mathcal{A}_{i},\rho_{i}\right)$ be a $\ast$-homomorphism, for some $\mathcal{A}_{i}\in \mathcal C$. Define  a unitary representation $f: G \to \mathcal U\left(%
{\displaystyle \prod_{i\in \mathbb N}^{\alpha}}
\left(  \mathcal{A}_{i},\rho_{i}\right)\right)$ by $f(g) = \phi(\delta(g)).$ Since $G$ is $\mathcal{C}$-stable, $f$ lifts to a unitary representation $f': G \to \mathcal U \left(\prod_{i\in \mathbb N}\mathcal A_i\right)$.
By the universal property of $C^*(G)$ there exists a $\ast$-homomorphism $\tilde f': C^*(G)\to \prod_{i\in \mathbb N}\mathcal A_i$ such that $\tilde f'(\delta(g)) = f'(g)$, for all $g\in G$. It implies that for any $a\in \mathbb C G$, $\tilde f'(a)$ is a lift of $\phi(a)$. Since $\mathbb C G$ is dense in $C^*(G)$, it implies that $\tilde f'$ is a lift of $\phi$.

Now assume $C^*(G)$ is $\mathcal{C}$-tracially stable and let $f: G \to \mathcal U\left(%
{\displaystyle \prod_{i\in \mathbb N}^{\alpha}}
\left(  \mathcal{A}_{i},\rho_{i}\right)\right)$ be a homomorphism, for some $\mathcal{A}_{i}\in \mathcal C$. By the universal property of $C^*(G)$ there exists a $\ast$-homomorphism $\tilde f: C^*(G) \to
%
{\displaystyle \prod_{i\in \mathbb N}^{\alpha}}
\left(  \mathcal{A}_{i},\rho_{i}\right)$ such that $\tilde f(\delta(g)) = f(g)$, for all $g\in G$. Since $C^*(G)$ is $\mathcal{C}$-tracially stable, we can lift $\tilde f$ to a $\ast$-homomorphism $\psi: C^*(G) \to \prod_{i\in \mathbb N} \mathcal A_i$. Then a homomorphism $f': G \to \mathcal U \left(\prod_{i\in \mathbb N}\mathcal A_i\right)$ defined by $f'(g) = \psi(\delta(g))$ will be a lift of $f$.
\end{proof}

\medskip

Recall that a $C^*$-algebra has {\it real rank zero} (RR0) if each self-adjoint element
can be approximated by self-adjoint elements with finite spectra.

In this paper the role of the class $\mathcal C$ will be played by the class of all matrix $C^*$-algebras, the class of all $II_1$-factors, the class of all von Neumann factors and the class of all $C^*$-algebras of real rank zero.

 Thus we will address matricial stability, $II_1$-factor stability, $W^*$-factor stability and $RR0$-stability for groups respectively. Since every
von Neumann algebra has real rank zero (\cite{BP}),  $RR0$-stability implies $W^*$-factor stability, and  of course $W^*$-stability implies both matricial and $II_1$-factor stability.

\medskip

From now on let $G$ be a discrete countable group.

\begin{theorem}
The classes of
matricially stable groups,  $II_1$-factor stable groups, $W^*$-factor stable groups, and $RR0$-stable groups are closed under finite free products and under the direct product with an abelian group.
\end{theorem}

\begin{proof}
This follows from [Th. 2.7 and Prop. 2.9 in \cite{MainPaper}]. (In fact Th.2.7 in \cite{MainPaper} is proved for the class of $\mathcal C$-tracially stable $C^*$-algebras, where the class $\mathcal C\subseteq RR0$ is closed under direct sums and unital corners, however for our proof it is sufficient that $\mathcal C$ is closed only under unital corners, and thus the theorem applies for matricial, $II_1$-factor and $W^*$-factor stability too).
\end{proof}

\medskip

Of course besides $W^*$-factor stability one also can introduce $W^*$-stability meaning liftings from tracial ultraproducts of (not necessarily factorial) von Neumann algebras.
In general we don't know if $W^*$-factor stability coincides with $W^*$-stability. However if a group is finitely presented, then  they coincide as we show below.
All necessary information about direct integrals and measurable cross-sections can be found in \cite{Arveson}.

\begin{theorem}\label{W*-stable} Let $G$ be a finitely presented group. Then $G$ is $W^*$-factor stable if and only if it is $W^*$-stable.
\end{theorem}
\begin{proof} We will give a proof for a group presented by one relation, because for finitely many relations it is absolutely similar.
So let  $G =\langle x_1, \ldots, x_s\;|\; \phi(x_1, \ldots, x_s) = 1\rangle$.
The "if" part is obvious, so let us assume that $G$ is $W^*$-factor stable. Then
for any $ \varepsilon>0$ there exists $\delta_0>0$ such that, for all factors $\left(
\mathcal{M},\tau \right)  $, for all $y_{1},\ldots,y_{s}%
\in \mathcal U(\mathcal{M})$ we have that if $\left \Vert \varphi \left(  y_{1}%
,\ldots,y_{s}\right) -1 \right \Vert _{2,\tau}<\delta_0,$ there is a homomorphism $\pi:G\rightarrow \mathcal U(M)$ such that%
\begin{equation}\label{2W*-stable}
\sum_{k=1}^{s}\left \Vert y_{k}-\pi \left(  x_{k}\right)  \right \Vert _{2,\tau
}^{2}<\varepsilon/37.
\end{equation}
We are going to prove that then for any $ \varepsilon>0$ there exists \begin{equation}\label{3W*-stable}\delta := \sqrt{\delta_0^2\frac{\epsilon}{37s}}\end{equation} such that, for all von Neumann algebras $\left(
\mathcal{M},\tau \right)  $, for all $y_{1},\ldots,y_{s}%
\in \mathcal U(\mathcal{M})$ we have that if $\left \Vert \varphi \left(  y_{1}%
,\ldots,y_{s}\right) -1 \right \Vert _{2,\tau}<\delta,$ there is a homomorphism $\pi:G\rightarrow \mathcal U(M)$ such that%
\[
\sum_{k=1}^{s}\left \Vert y_{k}-\pi \left(  x_{k}\right)  \right \Vert _{2,\tau
}^{2}<\varepsilon.
\]
So let  $\left(
\mathcal{M},\tau \right)  $ be a von Neumann algebra,  $y_{1},\ldots,y_{s}%
\in \mathcal U(\mathcal{M})$,
\begin{equation}\label{1W*-stable}\left \Vert \varphi \left(  y_{1}%
,\ldots,y_{s}\right) -1 \right \Vert _{2,\tau}<\delta.
\end{equation} Without loss of generality we can assume that $\tau$ is faithful and also we can replace $\mathcal M$
  with $W^{\ast}\left(  y_{1},\ldots,y_{s}\right)
$, so we can assume $\mathcal{M}=W^{\ast}\left(  y_{1},\ldots,y_{s}\right)  $.
Then $\mathcal{M}$ acts faithfully on $L^{2}\left(  W^{\ast}\left(
y_{1},\ldots,y_{s}\right)  ,\tau \right)  $, which is a separable Hilbert
space. Thus we can write%
\[
\mathcal{M}=\int_{\Omega}^{\oplus}\mathcal{M}_{\omega}d\mu \left(
\omega \right)
\]
for some probability space $\left(  \Omega,\mu \right)  $, where each
$\mathcal{M}_{\omega}$ is a factor von Neumann algebra with a unique faithful
normal tracial state $\tau_{\omega}$, and such that, for every $y=\int
_{\Omega}^{\oplus}y\left(  \omega \right)  d\mu \left(  \omega \right)
\in \mathcal{M}$, we have%
\[
\tau \left(  y\right)  =\int_{\Omega}\tau_{\omega}\left(  y\left(
\omega \right)  \right)  d\mu \left(  \omega \right)  \text{.}%
\]
Hence%
\[
\left \Vert y\right \Vert _{2,\tau}^{2}=\tau \left(  y^{\ast}y\right)
=\int_{\Omega}\left \Vert y\left(  \omega \right)  \right \Vert _{2,\tau_{\omega
}}^{2}d\mu \left(  \omega \right)  .
\]
Let $$E=\left \{  \omega \in \Omega:\left \Vert \varphi \left(  y_{1}\left(
\omega \right)  ,\ldots,y_{s}\left(  \omega \right)  \right) -1 \right \Vert
_{2,\tau_{\omega}}\geq \delta_{0}\right \}.$$ Then
\begin{multline*}\left \Vert \varphi \left(  y_{1}%
,\ldots,y_{s}\right) -1 \right \Vert _{2,\tau}^2 = \int_{\Omega} \left \Vert \varphi \left(  y_{1}\left(
\omega \right)  ,\ldots,y_{s}\left(  \omega \right)  \right) -1 \right \Vert
_{2,\tau_{\omega}}^2 d\mu(w) \\ \ge \int_{E} \left \Vert \varphi \left(  y_{1}\left(
\omega \right)  ,\ldots,y_{s}\left(  \omega \right)  \right) -1 \right \Vert
_{2,\tau_{\omega}}^2 d\mu(w) \ge \delta_0^2\mu(E).\end{multline*}
Using (\ref{2W*-stable}) and (\ref{3W*-stable}), it follows that
 $$\mu \left(  E\right)
\leq \frac{1}{\delta_{0}^{2}}\left \Vert \varphi \left(  y_{1}  ,\ldots,y_{s}  \right) -1 \right \Vert
_{2,\tau}^2   <\frac{\varepsilon}{37s}.$$
 For each $\omega \in E$, we define $\pi_{\omega}: G\rightarrow
\mathcal U(\mathcal{M}_{\omega})$ by $\pi_{\omega}\left(  g\right)  = 1$. Then, for $\omega \in E$,
\[
\sum_{k=1}^{s}\left \Vert y_{k}\left(  \omega \right)  -\pi_{\omega}\left(
x_{k}\right)  \right \Vert _{2,\tau_{\omega}}^{2}\leq \sum_{k=1}^{s}4=4s.
\]
Hence
\[
\int_{E}\sum_{k=1}^{s}\left \Vert y_{k}\left(  \omega \right)  -\pi_{\omega
}\left(  x_{k}\right)  \right \Vert _{2,\tau_{\omega}}^{2}\leq4s\mu \left(
E\right)  <\frac{4\varepsilon}{37}.
\]
By $W^*$-factor stability of $G$, (\ref{2W*-stable}), for each $\omega \in \Omega \backslash E,$ there is a representation $\pi_{\omega}%
:G\rightarrow \mathcal U(\mathcal{M}_{\omega})$ so that
\[
\sum_{k=1}^{s}\left \Vert y_{k}-\pi_{\omega}\left(  x_{k}\right)  \right \Vert
_{2,\tau_{\omega}}^{2}<\varepsilon/37
\]

Standard measurable cross-section theorems allow us to choose $\pi_{\omega}$
so that, for every $g\in G$, the map $g\mapsto \pi_{\omega}\left(
g\right)  $ is weak* measurable. Define a representation $\pi: G\rightarrow
\mathcal U(\mathcal{M})$ by%
\[
\pi \left(  g\right)  =\int_{\Omega}^{\oplus}\pi_{\omega}\left(  g\right)
d\mu \left(  \omega \right)  .
\]
Then%
\[
\sum_{k=1}^{s}\left \Vert y_{k}-\pi \left(  x_{k}\right)  \right \Vert _{2,\tau
}^{2}=\sum_{k=1}^{s}\int_{\Omega}\left \Vert y_{k}\left(  \omega \right)
-\pi_{\omega}\left(  x_{k}\right)  \right \Vert _{2,\tau}^{2}=
\]%
\[
\sum_{k=1}^{s}\int_{E}\left \Vert y_{k}\left(  \omega \right)  -\pi_{\omega
}\left(  x_{k}\right)  \right \Vert _{2,\tau}^{2}+\sum_{k=1}^{s}\int
_{\Omega \backslash E}\left \Vert y_{k}\left(  \omega \right)  -\pi_{\omega
}\left(  x_{k}\right)  \right \Vert _{2,\tau}^{2}\leq
\]%
\[
4\varepsilon/37+\int_{\Omega \backslash E}\varepsilon/37d\mu \left(
\omega \right)  <\varepsilon.
\]

\end{proof}

\noindent {\bf Remark.} \textrm{Using noncommutative continuous functions \cite{DonNCfunctions}, one can rewrite this proof to show that any finitely generated $C^*$-algebra which has a unital 1-dimensional representation is $W^*$-factor tracially stable if and only if it is $W^*$ tracially stable. In particular Theorem \ref{W*-stable} holds for any finitely generated group, not necessarily finitely presented.}

\section{A necessary condition for matricial stability and a characterization of matricial stability for amenable groups}

Recall that a \textit{character} of a group $G$ is a positive definite
function on $G$ which is constant on conjugacy classes and takes value 1 at
the unit.

We will say that a character $\tau$ is {\it embeddable} if it factorizes through a homomorphism to a tracial ultraproduct of matrices,
that is if there is a non-trivial ultrafilter $\alpha$ on $\mathbb{N}$ and a
homomorphism $f: G \to \mathcal U\left( %
{\displaystyle \prod_{n\in \mathbb{N}}^{\alpha}}
\left(  \mathcal{M}_{n}(\mathbb C),tr_{n}\right)\right)  $ such that $\tau_{\alpha}\circ f =\tau$.

This definition is analogous to the definition of embeddable trace on a $C^*$-algebra (see \cite{MainPaper}).
On an amenable group every character is embeddable.
If Connes' embedding conjecture holds, then on any group every character is embeddable.

\medskip

The following easy statement gives a necessary condition for matricial stability.

\begin{theorem}
\label{NecessaryConditionGroups} If $G$ is matricially stable, then
each embeddable character of $G$ is a pointwise limit of traces of finite-dimensional representations.
\end{theorem}

\begin{proof}
Let $\tau$ be an embeddable character on $G$. Then there is a non-trivial ultrafilter $\alpha$ on $\mathbb{N}$ and a
homomorphism $f: G \to \mathcal U\left( %
{\displaystyle \prod_{n\in \mathbb{N}}^{\alpha}}
\left(  \mathcal{M}_{n}(\mathbb C),tr_{n}\right)\right)  $ such that
\begin{equation}\label{EmbTrace}\tau_{\alpha}\circ f =\tau.\end{equation}
By matricial stability of $G$, there exists homomorphisms $f_n: G \to \mathcal U \left(\mathcal{M}_{n}(\mathbb C)\right)$ such that $f(g) = \{f_n(g)\}_{\alpha}$. Together with (\ref{EmbTrace}) it  implies that
\begin{equation}\label{LimitTrace}\tau(g) =\lim_{\alpha} tr_n (f_n(g)),\end{equation} for all $g\in G$. It easily implies that there is a subsequence $n_j$ such that \begin{equation}\label{LimitTrace2}\tau(g) =\lim_{j\to \infty} tr_{n_j} (f_{n_j}(g)),\end{equation} for all $g\in G$. Indeed, since $G$ is countable, we list all its elements as $g_1, g_2, \ldots$ and then by (\ref{LimitTrace}) the set
 $\{n\in \mathbb N\;|\; |\tau(g_1) - tr_n(f_n(g_1))|<1/2\}$ is in $\alpha$ and hence is not empty. So there is $n_1$ such that $$|\tau(g_1) - tr_{n_1}(f_{n_1}(g_1))|<1/2.$$ We continue inductively.
 Suppose $n_1< n_2< \ldots < n_{k-1}$ such that $$|\tau(g_i) - tr_{n_l}(f_{n_l}(g_i))|<\frac{1}{2^l},$$ $i = 1, \ldots, l$, $l = 2, \ldots, k-1$,  are already found. The set
 \begin{multline*}\{n\in \mathbb N\;|\; n> n_{k-1}, \; |\tau(g_i) - tr_{n}(f_{n}(g_i))|<\frac{1}{2^k}, \; i = 1, \ldots, k\}\\ =  \{n\in \mathbb N\;|\; n> n_{k-1}\} \bigcap \left(\bigcap_{i\le k} \{n\in \mathbb N\;|\; |\tau(g_i) - tr_{n}(f_{n}(g_i))|<\frac{1}{2^k}\}\right)\end{multline*} is in $\alpha$ and hence is not empty. Thus there is $n_k> n_{k-1}$ such that $$|\tau(g_i) - tr_{n_k}(f_{n_k}(g_i))|<\frac{1}{2^k},$$ $i = 1, \ldots, k$.

 Now the statement follows from (\ref{LimitTrace2}).
\end{proof}

The next 2 statements are corollaries of our results in \cite{MainPaper}. The first of them gives  a complete characterization of matricial stability and of $W^*$-factor stability for amenable groups. $II_1$-factor stability is automatic for amenable groups.

\begin{theorem}
\label{amenablegroup} Let $G$ be an amenable group. The following are equivalent:

\begin{enumerate}
\item $G$ is matricially stable

\item $G$ is $W^*$-factor  stable.

\item Each character of $G$ is a pointwise limit of traces of finite-dimensional representations.
\end{enumerate}
\end{theorem}

\begin{proof}
 As
is well known, a positive definite function on $G$ extends in unique way to a
state on $C^{*}(G)$ (see e.g. \cite{Davidson}, p.188), and it is obvious
that a positive definite function is constant on conjugacy classes if and only
if the corresponding state is a trace. Thus (embeddable) characters of $G$ are in 1-to-1 correspondence with (embeddable) tracial states on $C^*(G)$ and the condition (3) is equivalent to the condition that for each tracial state $\tau$ on $C^*(G)$ there are finite-dimensional representations $\pi_n$ of $C^*(G)$ such that $$\tau(a) = \lim_{n\to \infty} tr \pi_n(a),$$ for each $a\in C^*(G)$.
Since for any group $G$, $C^*(G)$ has a one-dimensional representation, the statement follows from [Theorem 3.8, \cite{MainPaper}].
\end{proof}

\begin{theorem}\label{VirtAb}
The class of $W^{*}$-factor stable groups contains all virtually
abelian groups.
\end{theorem}

\begin{proof}
As is well known, $G$ is virtually abelian if and only if  $C^{*}(G)$ is GCR  (\cite{Thoma1}, \cite{Thoma2}).
Since $C^*(G)$ has a 1-dimensional
representation, the
statement follows from [Corollary 3.9, \cite{MainPaper}].
\end{proof}

We will use Theorem \ref{amenablegroup} to prove that the discrete Heisenberg group is $W^*$-factor stable.
Recall that the discrete Heisenberg group $\mathbb{H}_{3}$ is the group generated by $u,v$
with the relations that $u$ and $v$ commute with $uvu^{-1}v^{-1}$. It is known
that $\mathbb{H}_{3} $ is amenable.

\begin{lemma}\label{extreme} If each extreme character is a pointwise limit of traces of finite-dimensional representations, then so is any character.
\end{lemma}
\begin{proof} Let $\tau$ be a character, $\epsilon >0$ and $g_1, \ldots, g_n\in G$. Since the set of all characters of a group is convex and compact in $\ast$-weak topology, there are rational numbers $s_1/m, \ldots, s_l/m$ with $s_1+ \ldots + s_l=m$, and extreme characters $\sigma_1, \ldots, \sigma_l$ such that
\begin{equation}\label{1lemma}|\tau(g_k) - \sum_{i=1}^l \frac{s_i}{m} \sigma_i(g_k)| \le \epsilon,\end{equation} $k =1, \ldots, n$. By the assumption, there exist representations $\pi_i: G \to M_{n_i}(\mathbb C)$ such that \begin{equation}\label{2lemma}|\sigma_i(g_k)  - tr \pi_i(g_k)| \le \epsilon.\end{equation} Let $L\in \mathbb N$ be such that $\frac{s_i}{n_i}L$ is an integer, for all $1\le i\le l$.
Let $$\pi = \oplus_{i=1}^l \pi_i^{\left(\frac{s_i}{n_i}L\right)}$$ (here $\pi_i^{\left(\frac{s_i}{n_i}L\right)}$  denotes a direct sum of $\frac{s_i}{n_i}L$ copies of $\pi_i$).
It is easy to check that \begin{equation}\label{3lemma}tr\pi(g_k) = \sum_{i=1}^l \frac{s_i}{m} tr \pi_i(g_k),\end{equation} $k=1, \ldots, n$.
By (\ref{1lemma}), (\ref{2lemma}), (\ref{3lemma}), $$|\tau(g_k) - tr \pi(g_k)|\le 2\epsilon.$$

\end{proof}

\begin{theorem}\label{Heisenberg}
$  \mathbb{H}_{3}  $ is $W^{*}$-factor stable.
\end{theorem}

\begin{proof}
Suppose $\tau$ is an extreme point in the set of characters of $\mathbb H_3$. Then it extends to an extreme
 tracial state on $C^*(\mathbb H_3)$, i.e.
 a factor tracial state on $C^*(\mathbb H_3)$. We will denote it also by $\tau$. Let $\pi:$ $C^{\ast}\left(  \mathbb{H}_{3}\right)
\rightarrow B\left(  H\right)  $ be the GNS representation for $\tau$. Let
$U=\pi \left(  u\right)  $ and $V=\pi \left(  v\right)  $. Since $\pi \left(
\mathcal{A}\right)  ^{\prime \prime}$ is a factor and $UVU^{-1}V^{-1}%
=\pi \left(  uvu^{-1}v^{-1}\right)  $ is in its center, there is a real number
$\theta$ such that
\[
UVU^{-1}=e^{2\pi i\theta}V
\]
and%
\[
V^{-1}UV=e^{2\pi i\theta}U
\]
First suppose $\theta$ is rational, then there is a positive integer $n$ such
that $n\theta \in \mathbb{Z}$. In this case we have%
\[
U^{n}VU^{-n}=V\text{ and }V^{-n}UV^{n}=U,
\]
which implies $U^{n}=\alpha$ and $V^{n}=\beta$ for scalars $\alpha$ and
$\beta$. For every positive integer $m$ there is a positive integer $k$ such
that $m<kn$. Thus%
\[
U^{-m}=U^{kn-m}U^{kn}=\alpha^{k}U^{kn-m}\text{ and }V^{-m}=\beta^{k}V^{kn-m}.
\]
Since $UV=e^{2\pi i\theta}VU$, every monomial in $U,V,U^{-1},V^{-1}$ can be
written as a scalar times $U^{a}V^{b}$ for integers $a,b$ with $0\leq a,b<n$.
Hence $C^{\ast}\left(  U,V\right)  $ is finite-dimensional, which means
$C^{\ast}\left(  U,V\right)  =C^{\ast}\left(  U,V\right)  ^{\prime \prime}$ is
isomorphic to $\mathcal{M}_{k}\left(  \mathbb{C}\right)  $ for some
$k\in \mathbb{N}$. Hence $\tau$ is a matricial tracial state.

Next suppose $\theta$ is irrational. Then $U, V$ give a representation of the
irrational rotation $C^{*}$-algebra $\mathcal{A}_{\theta}$. Since
$\mathcal{A}_{\theta}$ is simple, $C^{\ast}\left(  U,V\right)  $
is isomorphic to $\mathcal{A}_{\theta}$ and hence has a unique tracial state.
In this case we can choose a sequence $\left \{  \theta_{k}\right \}  $ of
rational numbers such that $\theta_{k} \to \theta$, and find finite-dimensional
irreducible representations $\pi_{k}: \mathbb H_3 \to M_{n_k}(\mathbb C)$  such that $\pi
_{k}\left(  uvu^{-1}v^{-1}\right)  =e^{2\pi i\theta_{k}}$. Let $\alpha$ be a non-trivial unltrafilter on $\mathbb N$, then in the tracial ultraproduct $\prod_{i\in \mathbb N}^{\alpha} M_{n_i}(\mathbb C)$ we get $\hat{U}=\left \{  \pi_{k}\left(  u\right)
\right \}_{\alpha}  $ and $\hat{V}=\left \{  \pi_{k}\left(  v\right)  \right \}_{\alpha}  $ satisfy
$\hat{U}\hat{V}\hat{U}^{-1}\hat{V}^{-1}=e^{2\pi i\theta}$. Thus C*$\left(
\hat{U},\hat{V}\right)  $ is also isomorphic to $\mathcal{A}_{\theta}$ and
hence has a unique tracial state which has to coincide with $\tau$. Hence, for every $a\in C^{\ast
}\left(  \mathbb{H}_{3}\right)  $%
\[
\tau \left(  a\right)  =\lim_{k\rightarrow \alpha}tr\left(  \pi
_{k}(a\right)  ).
\]
It follows from Lemma \ref{extreme} and Theorem \ref{amenablegroup} that $ \mathbb{H}_{3}  $
is $W^{*}$-factor stable.
\end{proof}

\noindent {\bf Remark.}
\textrm{It would be interesting to know if our characterization of matricial
 stability for amenable groups can be reformulated in terms of
"separation properties" of groups. By this we mean properties like residual
finiteness (which means that a group has a
separating family of homomorphisms into finite groups), the property of being maximally almost periodic (which means that a group has a
separating family of finite-dimensional representations), the property of
being conjugacy separable (which means that homomorphisms to finite groups separate conjugacy classes), the property that finite-dimensional
representations separate conjugacy classes, etc. For example, it is easy to
see that for an amenable group matricial stability implies that the
group is maximally almost periodic. We don't know if it is also a sufficient
condition, and we believe that it is not. Otherwise for $C^*(G)$ to be nuclear and matricially tracially  stable would be equivalent to be nuclear RFD
(since an amenable group $G$ is maximally almost periodic iff $C^*(G)$ is RFD by \cite{BL})
and in in \cite{MainPaper} we constructed an example of  nuclear RFD $C^*$-algebra which is not matricially tracially stable. This makes us think
that for an amenable group being maximally almost periodic is probably not sufficient for matricial stability.
Separation properties for conjugacy classes seem to us to be
more relevant. For instance if a group is conjugacy separable, then the
Stone-Weierstrass theorem leads to an easy proof that each character of $G$ is
a pointwise limit of linear combinations of two traces of finite-dimensional
representations (which is close to the condition 3) in Theorem
\ref{amenablegroup}). In the opposite direction, by Theorem
\ref{amenablegroup} the property that finite-dimensional representations
separate conjugacy classes would be necessary if the characters separate
conjugacy classes. }

\medskip

\textbf{Question}: Let $G$ be an amenable maximally almost periodic group. Do its characters separate
conjugacy classes?

\section{Character rigidity and the approximation property ($\ast$).}

Below we will say that a group $G$ has {\it the approximation property ($\ast$)} if any embeddable character of $G$  is a pointwise limit of traces of finite-dimensional representations. Thus by Theorem \ref{NecessaryConditionGroups}, the approximation property ($\ast$) is necessary for being matricially stable, and by Theorem \ref{amenablegroup}, if a group is  amenable, then it is also sufficient.

Following \cite{PetersonThom} (also \cite{Bekka}) we will say that a character is {\it induced from the center} if it vanishes outside the center.

An example of a character induced from the center is a character $\delta_{e}$ defined by $\delta_{e}(g) =
\left \{
\begin{array}
[c]{cc}%
1 & \text{if }g=e\\
0 & \text{if } g\neq e.
\end{array}
\right.  $

\begin{proposition}
\label{CanonicalCharacter} Let $G$ be a maximally almost periodic group. Then
$\delta_{e}$ is a pointwise limit of traces of some finite-dimensional
representations of $G$.
\end{proposition}
\begin{proof}
Let $\epsilon>0$, $g_{1}, \ldots, g_{n} \in G$, $g_{i} \neq1$ for all $i=1,
\ldots, n$. Since $G$ is maximally almost periodic, we can find a
finite-dimensional representation $\pi$ such that $\pi(g_{i}) \neq1$ for all
$i=1, \ldots, n$. Let $\chi: G \to \mathbb{C}$ be the trivial representation,
$\tilde \pi= \pi \oplus \chi$. Then
\[
|tr \tilde \pi(g_{i})| = \left| \frac{(dim \pi)tr \pi(g_{i}) +1}{dim \pi+
1}\right|  < 1,
\]
since this is absolute value of the average of numbers of absolute value not
larger than 1, not all of which are equal.

Let $\tilde \pi^{\otimes N}$ be the N-th tensor power of the representation
$\tilde \pi$. Then
\[
tr \tilde \pi^{\otimes N} (g_{i}) = (tr \tilde \pi(g_{i}))^{N} < \epsilon
\]
if $N$ is big enough. Thus
\[
|tr \tilde \pi^{\otimes N}(g_{i}) - \delta_{e}(g_{i})| < \epsilon
\]
for $i = 1, \ldots, n$.
\end{proof}

Below we will show that when a group is residually finite (RF), the approximation property above holds not only for $\delta_e$ but for all characters
induced from the center.

\begin{lemma}\label{centerFinGen} Let $G$ be a RF group and suppose its center $Z(G)$ is finitely generated.  Let $g_1, \ldots, g_N \in Z(G)$, $g_i \neq g_j$ when $i\neq j$.  Let $H$ be the subgroup generated by $g_1, \ldots, g_N$ and  $\chi$ be a 1-dimensional representation of $H$. Let $g_1', \ldots, g'_{m} \notin Z(G)$ and let $\epsilon > 0$. Then there exists a finite group $G_0$, a surjective homomorphism $f: G \to G_0$ and a 1-dimensional representation $\tilde \chi$ of  $f(H)$ such that $$|\tilde \chi(f(g_i)) - \chi(g_i)| < \epsilon,$$ for $i = 1, \ldots, N$ and $f(g_i')\notin f(Z(G))$, for $i=1, \ldots, N'$.
\end{lemma}
\begin{proof}
Since $H$ is a finitely generated abelian group, it can be written  as $$H = \mathbb Z^{s}\times \Gamma,$$ where $s\in \mathbb N$ and $\Gamma$ is a finite abelian group.
So we can write $g_j = (n_1^j, n_2^j, \ldots, n_s^j, t)$ with $n_i^j\in \mathbb Z$, $t\in \Gamma$, $j\le N$. Let $\mathbb Z^{(i)}$ denote the i-th copy of $\mathbb Z$ in $H$.
For each $i\le s$ there is $\theta_i$ such that
\begin{equation}\label{RF2}\chi |_{\mathbb Z^{(i)}}(n) = e^{2\pi i n \theta_i}.\end{equation}
Let $$L^{(i)} = \max_{j\le N} |n_i^{j}|,$$ $i = 1, \ldots, s$. For each $i\le s$ there exists  $k_{0, i}$ such that for any $k\ge k_{0, i}$, the k-th roots of unity form an $\frac{\epsilon}{s(L^{(i)} +1)}$-net in the unit circle.

Since $g_1', \ldots, g'_{m} \notin Z(G)$, there exist $g''_1, \ldots, g''_{m}\in G$ such that $g_i'g''_i \neq g''_ig_i'$, $i = 1, \ldots, m$.
Since $G$ is RF, there is a finite group $G_0$ and a surjective homomorphism $f: G \to G_0$ such that \begin{equation}\label{RF5}f(g_i'g''_i)\neq f(g''_ig_i'),\end{equation} for $i = 1, \ldots, m$ and  \begin{equation}\label{RF1}f(n_1, \ldots, n_s, t) \neq f(n_1', \ldots, n_s', t'),\end{equation} when $t\in \Gamma$, $n_i, n_i' \le k_{0, i}$ and the tuples $(n_1, \ldots, n_s, t)$ and $(n_1', \ldots, n_s', t')$ do not coincide. It follows from (\ref{RF5}) that \begin{equation}\label{RF6} f(g_i')\notin f(Z(G)),\end{equation} $i =1, \ldots, m$.

 It is easy to see that $f(H) \cong (\prod_{i\le s}f(\mathbb Z^{(i)})) \times f(\Gamma)$. Hence $$f(H) = \mathbb Z_{k_1}\times \ldots \times \mathbb Z_{k_s} \times \tilde\Gamma, $$ for some $k_1, \ldots, k_s \in \mathbb N$ and some finite abelian group $\tilde\Gamma$. It follows from (\ref{RF1}) that $k_i \ge k_{0, i}$ and that $|\tilde \Gamma| \ge |\Gamma|$. Since $\tilde \Gamma$ is a homomorphic image of $\Gamma$, the latter implies that $\tilde \Gamma \cong \Gamma$. The first inequality, $k_i \ge k_{0, i}$, implies that there is $l_i< k_i$ such that
\begin{equation}\label{RF3} |e^{2\pi i l_i/k_i} - e^{2 \pi i \theta_i}| \le \frac{\epsilon}{s(L^{(i)}+1)}.\end{equation}

 Define  a 1-dimensional representation $\tilde \chi_i$ of $\mathbb Z_{k_i}$ by $$\tilde\chi_i(m) = e^{2\pi im l_i/k_i},$$ for each $m\in \mathbb Z_{k_i}$. Using (\ref{RF3}),  for any $m\le L^{(i)}$ we easily  obtain by induction that $$|\tilde\chi_i(m) - \chi |_{\mathbb Z^{(i)} } (m) | =
|e^{2\pi i(m \mod k_i)l_i/k_i} - e^{2\pi i m \theta_i}| = |e^{2\pi i m l_i/k_i} - e^{2\pi i m \theta_i}| \le \frac{\epsilon (m+1) }{s(L^{(i)}+1)}.$$
In particular for any $m\le L^{(i)}$ we obtain \begin{equation}\label{RF4}|\tilde\chi_i(m) - \chi |_{\mathbb Z^{(i)} } (m) | \le \frac{\epsilon  }{s}.\end{equation}

Define a 1-dimensional representation $\tilde \chi$ of $f(H)$
by $$\tilde\chi(f(n_1, \ldots, n_s, t)) = \tilde\chi_1(n_1)\ldots\tilde\chi_s(n_s)\chi(t),$$ for all $n_i \in \mathbb Z, t\in \Gamma$. From (\ref{RF4}) we deduce  ( estimating  $|a_1\ldots a_s - b_1\ldots b_s|$ in a standard way) that for any $n_i \le L^{(i)}$, $t\in \Gamma$ $$|\tilde\chi(f(n_1, \ldots, n_s, t)) - \chi(n_1, \ldots, n_s, t)| \le \epsilon.$$ Hence $$|\tilde \chi(f(g_i)) - \chi(g_i)| < \epsilon,$$ for $i = 1, \ldots, N$. This, together with (\ref{RF6}), completes the proof.
\end{proof}

\begin{theorem}\label{CharacterRigidity} Suppose $G$ is RF.  Then each character of $G$ induced from the center of $G$ is a pointwise limit of traces of
finite-dimensional representations.
\end{theorem}

\begin{proof}  By Lemma \ref{extreme} it will be sufficient to prove that each extreme point of the set of all characters induced from $Z(G)$ is a pointwise limit of traces of
finite-dimensional representations. Since an extreme point of the set of characters of an abelian group is a 1-dimensional representation, we should prove that if $\chi|_{Z(G)}$ is  a 1 -dimensional representation and $\chi$ vanishes outside $Z(G)$, then $\chi$ is a pointwise limit of traces of
finite-dimensional representations. Let $g_1, \ldots, g_N \in Z(G)$, $ g'_1, \ldots, g'_m\notin Z(G)$, $\epsilon > 0$. We need to find a finite-dimensional representation $\pi$ of $G$ such that  $|\chi(g_i) - tr(\pi(g_i))|\le \epsilon$, $i=1, \ldots, N$, and $|\chi(g'_i) - tr(\pi(g'_i))|\le \epsilon$, $i=1, \ldots, m$. Let $H$ be the subgroup generated by $g_1, \ldots, g_N$. By Lemma \ref{centerFinGen} there is a finite group $G_0$, a surjective homomorphism $f: G \to G_0$ and a 1-dimensional representation $\tilde\chi$ of $f(H)$ such that \begin{equation}\label{ChRig1}|\tilde\chi(f(g_i)) - \chi(g_i)|\le \epsilon,\end{equation} $i = 1, \ldots, N$, and \begin{equation}\label{ChRig2} f(g'_i)\notin f(Z(G)),\end{equation} $i =1, \ldots, m$.
 Let $\tilde \pi$ be the representation of $G_0$ induced from the 1-dimensional representation $\tilde\chi$ of $f(H)$. By Frobenius formula
 $$ tr \tilde \pi(f(g)) = \sum_{x\in G_0/f(H)} \hat\chi(x^{-1}f(g)x),$$ where
 \begin{equation}\label{ChRig3}
  \hat\chi(k) = \left\{
     \begin{array}{lcl}
       \tilde\chi(k) &;& k\in f(H)\\
       &&\\
       0 &;&  k\notin f(H).
     \end{array}
   \right.
\end{equation}

Since $f(H)$ is a central subgroup of $G_0$, it implies easily that

\begin{equation}\label{ChRig3}
   tr \tilde \pi(f(g)) = \left\{
     \begin{array}{lcl}
       \tilde\chi(f(g)) &;& f(g)\in f(H)\\
       &&\\
       0 &;&  f(g)\notin f(H).
     \end{array}
   \right.
\end{equation}

Let $\pi = \tilde \pi\circ f$. Then, by (\ref{ChRig1}), (\ref{ChRig2}) and (\ref{ChRig3}),   for each $i\le N$ $$|\chi(g_i) - tr(\pi(g_i))| = |\chi(g_i) - \tilde\chi(f(g_i))|
\le \epsilon, $$ and for each $i\le m$ $$|\chi(g_i') - tr(\pi(g_i'))| = 0.$$
\end{proof}

A group $G$ has {\it character rigidity} if the only extremal characters of $G$ which are not induced from the center
of $G$ are the traces of finite-dimensional representations (\cite{PetersonThom}).



\begin{corollary}\label{RFCharRig} If $G$ is RF and has character rigidity, then $G$ has the approximation property ($\ast$).
\end{corollary}

As was proved by Bekka \cite{Bekka}  $SL_{3}(\mathbb{Z})$ has character rigidity. Thus, by Corollary \ref{RFCharRig},  the necessary condition for matricial stability from  Theorem \ref{NecessaryConditionGroups} holds.  Since $SL_{3}(\mathbb{Z})$ is
non-amenable, we don't know if it is also sufficient.

\medskip

\textbf{Question}: Is $SL_{3}(\mathbb{Z})$ matricially stable?

\medskip

\section{One-relator groups with center.}

Recall that a \textit{one-relator group} is a group $G$ with a presentation
$G=\langle S | R\rangle$ where the generating set $S$ is finite and $R$ is a
single word on $S^{\pm1}$. All 1-relator groups but the Baumslag-Solitar groups
$BS(1, m)$ are non-amenable (\cite{Grigorchuk}).

We are going to prove that any one-relator group with a non-trivial center is
$II_{1}$-factor stable. All such groups are known to be residually
finite (\cite{Dyer}).

It was shown in \cite{Pietrowski1} that every such non-cyclic group is
presentable in one of two ways: as
\begin{equation}
\label{1-relatorgroupCase1}G= \left<  x_{1}, \ldots, x_{n} \;| \; x_{1}%
^{a_{1}} = x_{2}^{b_{1}}, x_{2}^{a_{2}} = x_{3}^{b_{2}}, \ldots,
x_{n-1}^{a_{n-1}} = x_{n}^{b_{n-1}}\right>
\end{equation}
where $a_{i}, b_{i}\ge2$ and $(a_{i}, b_{j}) = 1$ for $i> j$ (when the
commutator quotient group is not free abelian of rank two); or as
\begin{equation}
\label{1-relatorgroupCase2}G= \left<  u, x_{1}, \ldots, x_{m} \;|
\;ux_{1}u^{-1} = x_{m}, \; x_{1}^{a_{1}} = x_{2}^{b_{1}},\; x_{2}^{a_{2}} =
x_{3}^{b_{2}}, \; \ldots,\; x_{m-1}^{a_{m-1}} = x_{m}^{b_{m-1}}\right>
\end{equation}
where $a_{i}, b_{i} \ge2$, $a_{1}\ldots a_{m-1} = b_{1} \ldots b_{m-1},
(a_{i}, b_{j}) = 1$ for $i> j$ (when the commutator quotient group is free
abelian of rank two).

Since cyclic groups are $II_1$-factor (even $RR0$) stable, we are left with the two cases
above. We are not going to use anywhere that $(a_{i}, b_{j}) = 1$.

\medskip

We will need a lemma from \cite{MainPaper} adjusted for the case of full group $C^*$-algebras.
It states that  pointwise $\left \Vert {\;}\right \Vert _{2}%
$-limits of  liftable homomorphisms are liftable.

\begin{lemma} (\cite{MainPaper}, Lemma 2.2)
\label{LimitOfLiftable} Suppose $G$ is a group, $\left \{  \left(  \mathcal{A}_{i}%
,\rho_{i}\right)  :i\in \mathbb N\right \}  $ is a family of tracial C*-algebras,
$\alpha$ is a nontrivial ultrafilter on $\mathbb N,$ and $\pi: G \rightarrow \mathcal U\left(%
{\displaystyle \prod_{i\in \mathbb N}^{\alpha}}
\left(  \mathcal{A}_{i},\rho_{i}\right)\right)  $ is a homomorphism
such that, for each $g\in G$,
\[
\pi \left(  g\right)  =\left \{  g\left(  i\right)  \right \}  _{\alpha
}.
\]
The following are equivalent:

\begin{enumerate}
\item $\pi$ is liftable

\item For every $\varepsilon>0$ and every finite subset $F\subset G$, there is a set
$E\in \alpha$ and for every $i\in E$ there is a homomorphism
$\pi_{i}:G\rightarrow \mathcal U\left(\mathcal{A}_{i}\right)$ such that, for every $g\in F$
and every $i\in E$,%
\[
\left \Vert \pi_{i}\left(  g\right)  -g\left(  i\right)  \right \Vert
_{2,\rho_{i}}<\varepsilon.
\]

\end{enumerate}
\end{lemma}

\subsection{Groups of the form (\ref{1-relatorgroupCase1}).}

\begin{lemma}
\label{projections} Suppose $\left \{  \left(  \mathcal{A}_{n},\rho_{n}\right)
\right \}  $ is a sequence of tracial $C^{*}$-algebras of real rank zero,
$\alpha$ is a non-trivial ultrafilter on $\mathbb{N}$, and $r_{1}, \ldots,
r_{N}, q\in%
{\displaystyle \prod_{n\in \mathbb{N}}^{\alpha}}
\left(  \mathcal{A}_{n},\rho_{n}\right)  $ are projections such that
$\sum_{i=1}^{N} r_{i} = q$. Suppose projections $Q_{n}\in \mathcal{A}_{n}$,
$n\in \mathbb{N}$ are such that $\left \{  Q_{n}\right \}  _{\alpha} = q$. Then
there exist projections $R_{i, n}\in \mathcal{A}_{n}$, $n\in \mathbb{N}$, $i =
1, \ldots, N$, such that $\left \{  R_{i, n}\right \}  _{\alpha} = r_{i}$, $i=
1, \ldots, N$, and $\sum_{i=1}^{N} R_{i, n} = Q_{n}.$
\end{lemma}
\begin{proof}
All $r_{i}$'s belong to the tracial ultraproduct $%
{\displaystyle \prod_{n\in \mathbb{N}}^{\alpha}}
\left(  Q_{n}\mathcal{A}_{n}Q_{n},\frac{1}{\rho_{n}\left(  Q_{n}\right)
}\rho_{n}\right)  $. $q$ is the unit element in this ultraproduct. Since
projections with sum 1 generate a commutative $C^{*}$-algebra, hence
$RR0$-stable by [Th. 2.5, \cite{MainPaper}], the statement follows.
\end{proof}

\begin{theorem}
\label{1-relator1} Let $G$ be as in (\ref{1-relatorgroupCase1}). Then $G$
 is $RR0$-stable.
\end{theorem}
\begin{proof}
To avoid notational nightmare we will prove $RR0$-tracial stability for the
case $G= \left<  x, y, z\;|\; x^{2} = y^{3}, y^{5} = z^{7}\right>  $, and the
proof for the general case is absolutely similar.

Suppose $\left \{  \left(  \mathcal{A}_{n},\rho_{n}\right)  \right \}  $ is a
sequence of tracial $C^{*}$-algebras of real rank zero, $\alpha$ is a
non-trivial ultrafilter on $\mathbb{N}$, and $X,Y,Z\in%
{\displaystyle \prod_{n\in \mathbb{N}}^{\alpha}}
\left(  \mathcal{A}_{n},\rho_{n}\right)  =_{\text{\textrm{def}}}\left(
\mathcal{A},\rho \right)  $ are unitary and $X^{2}=Y^{3},Y^{5}=Z^{7}$. Then
$X^{10}=Y^{15}=Z^{21}=_{\text{\textrm{def}}}W$. We can write $X=\left \{
X_{n}\right \}  _{\alpha}$, $Y=\left \{  Y_{n}\right \}  _{\alpha}$, $Z=\left \{
Z_{n}\right \}  _{\alpha}$. Suppose $\varepsilon>0$. Since $X$, $Y$ and $Z$
commute with $W$, they commute with every spectral projection of $W$, and
since $\mathcal{A}$ is a von Neumann algebra, the spectral projections of $W$
are in $\mathcal{A}$. We can choose an orthogonal family of nonzero spectral
projections $\left \{  P_{1},\ldots,P_{s}\right \}  $ of $W$ whose sum is $1$
and we can choose $\lambda_{1},\ldots,\lambda_{s}\in \mathbb{T}$ such that if
 $\Omega=\sum_{k=1}^{s}\lambda_{k}P_{k}$, then%
\[
\left \Vert W^{m}-\Omega^{m}\right \Vert _{2}<\varepsilon,
\]
for $m\in \left \{  1,1/10,1/15,1/21\right \}  .$ Here and below by $W^{1/10}, W^{1/15}$, etc., we mean the normal operators obtained
by applying the Borel functions
$z^{1/10} =_{def} |z|^{1/10}e^{\frac{iArg z}{10}}$, etc.,  to $W$.

  Let $X^{\prime}=XW^{-1/10}$,
$Y^{\prime}=YW^{-1/15}$, $Z^{\prime}=ZW^{-1/21}$. Then $X^{\prime},$
$Y^{\prime}$ and $Z^{\prime}$ are unitary and
\begin{equation}
\left(  X^{\prime}\right)  ^{2}=(Y^{\prime})^{3} \label{powers1}%
\end{equation}%
\begin{equation}
(Y^{\prime})^{5}=(Z^{\prime})^{7}, \label{powers2}%
\end{equation}%
\begin{equation}
(X^{\prime})^{10}=(Y^{\prime})^{15}=(Z^{\prime})^{21}=1. \label{powers=1}%
\end{equation}

Moreover,
\begin{equation}
\label{estimate1}\| X-X^{\prime}\Omega^{1/10}\|_{2}<\varepsilon,
\end{equation}
\begin{equation}
\label{estimate2}\|Y-Y^{\prime}\Omega^{1/15}\|_{2}<\varepsilon,
\end{equation}
\begin{equation}
\label{estimate3}\|Z-Z^{\prime}\Omega^{1/21}\|_{2}<\varepsilon.
\end{equation}
Clearly
\[
T=\sum_{k=1}^{s}P_{k}TP_{k}%
\]
for $T\in \left \{  X,Y,Z, X^{\prime}, Y^{\prime}, Z^{\prime}, W, \Omega\right \}  .$
For each $n$ we can find an orthogonal family $\left \{  P_{n,1},\ldots
,P_{n,s}\right \}  $ of projections in $\mathcal{A}_{n}$ whose sum is $1$ such
that, for $1\leq k\leq s,$%
\[
P_{k}=\left \{  P_{n,k}\right \}  _{\alpha}\text{ .}%
\]
It is clear that $\sum_{k=1}^{s}P_{k}\mathcal{A}P_{k}$ is the tracial
ultraproduct $%
{\displaystyle \prod_{n\in \mathbb{N}}^{\alpha}}
\left(  \sum_{k=1}^{s}P_{n,k}\mathcal{A}_{n}P_{n,k},\rho_{n}\right)  $ and
that each $\left(  P_{k}\mathcal{A}P_{k},\frac{1}{\rho \left(  P_{k}\right)
}\rho \right)  $ is the tracial ultraproduct $%
{\displaystyle \prod_{n\in \mathbb{N}}^{\alpha}}
\left(  P_{n,k}\mathcal{A}_{n}P_{n,k},\frac{1}{\rho_{n}\left(  P_{n,k}\right)
}\rho_{n}\right)  $. By (\ref{powers=1}) $X^{\prime},Y^{\prime},Z^{\prime}$
can be written in the form
\[
X^{\prime}=\sum_{j=1}^{10}e^{\frac{2\pi ij}{10}}q_{j},\;Y^{\prime}=\sum
_{j=1}^{15}e^{\frac{2\pi ij}{15}}r_{j},\;Z^{\prime}=\sum_{j=1}^{21}%
e^{\frac{2\pi ij}{21}}s_{j},
\]
where $\{q_{j}\},\{r_{j}\},\{s_{j}\}$ are families of projections in
$\sum_{k=1}^{s}P_{k}\mathcal{A}P_{k}$ which sum to $1$. It is easy to see that
(\ref{powers1}) is equivalent to the system of equations
\[
q_{1}+q_{6}=r_{1}+r_{6}+r_{11}%
\]%
\[
q_{2}+q_{7}=r_{2}+r_{7}+r_{12}%
\]%
\[
q_{3}+q_{8}=r_{3}+r_{8}+r_{13}%
\]%
\[
q_{4}+q_{9}=r_{4}+r_{9}+r_{14}%
\]%
\[
q_{5}+q_{10}=r_{5}+r_{10}+r_{15}%
\]
and (\ref{powers2}) is equivalent to the system of equations
\[
r_{1}+r_{4}+r_{7}+r_{10}+r_{13}=s_{1}+s_{4}+s_{7}+s_{10}+s_{13}+s_{16}+s_{19}%
\]%
\[
r_{2}+r_{5}+r_{8}+r_{11}+r_{14}=s_{2}+s_{5}+s_{8}+s_{11}+s_{14}+s_{17}+s_{20}%
\]%
\[
r_{3}+r_{6}+r_{9}+r_{12}+r_{15}=s_{3}+s_{6}+s_{9}+s_{12}+s_{15}+s_{18}%
+s_{21}.
\]
Since $C^{\ast}(q_{1},\ldots,q_{10})$ is commutative, it is $RR0$-stable, so we can find projections $Q_{1}=\left \{  Q_{n, 1}\right \}
,\ldots,Q_{10}=\left \{  Q_{n, 10}\right \}  \in%
{\displaystyle \prod_{n\in \mathbb{N}}^{\alpha}}
\left(  \sum_{k=1}^{s}P_{n,k}\mathcal{A}_{n}P_{n,k}\right)  $ with sum 1 such
that $q_{1}=\left \{  Q_{n, 1}\right \}  _{\alpha},\ldots,q_{10}=\left \{
Q_{n, 10}\right \}  _{\alpha}$. By Lemma \ref{projections} we can find
projections $R_{1}=\left \{  R_{n, 1}\right \}  ,\ldots,R_{15}=\left \{
R_{n, 15}\right \}  \in%
{\displaystyle \prod_{n\in \mathbb{N}}^{\alpha}}
\left(  \sum_{k=1}^{s}P_{n,k}\mathcal{A}_{n}P_{n,k}\right)  $ such that
$r_{1}=\left \{  R_{n, 1}\right \}  _{\alpha},\ldots,r_{15}=\left \{
R_{n, 15}\right \}  _{\alpha}$ and
\[
Q_{1}+Q_{6}=R_{1}+R_{6}+R_{11}%
\]%
\[
Q_{2}+Q_{7}=R_{2}+R_{7}+R_{12}%
\]%
\[
Q_{3}+Q_{8}=R_{3}+R_{8}+R_{13}%
\]%
\[
Q_{4}+Q_{9}=R_{4}+R_{9}+R_{14}%
\]%
\[
Q_{5}+Q_{10}=R_{5}+R_{10}+R_{15}.
\]
Again by Lemma \ref{projections} we can find projections $S_{1}=\left \{
S_{n, 1}\right \}  ,\ldots,S_{21}=\left \{  S_{n, 21}\right \}  \in%
{\displaystyle \prod_{n\in \mathbb{N}}^{\alpha}}
\left(  \sum_{k=1}^{s}P_{n,k}\mathcal{A}_{n}P_{n,k}\right)  $ such that
$s_{1}=\left \{  S_{n, 1}\right \}  _{\alpha},\ldots,s_{15}=\left \{
S_{n, 15}\right \}  _{\alpha}$ and
\[
R_{1}+R_{4}+R_{7}+R_{10}+R_{13}=S_{1}+S_{4}+S_{7}+S_{10}+S_{13}+S_{16}+S_{19}%
\]%
\[
R_{2}+R_{5}+R_{8}+R_{11}+R_{14}=S_{2}+S_{5}+S_{8}+S_{11}+S_{14}+S_{17}+S_{20}%
\]%
\[
R_{3}+R_{6}+R_{9}+R_{12}+R_{15}=S_{3}+S_{6}+S_{9}+S_{12}+S_{15}+S_{18}%
+S_{21}.
\]
Let
\[
X_{n}^{\prime}=\sum_{j=1}^{10}e^{\frac{2\pi ij}{10}}Q_{n, j},\;Y_{n}^{\prime
}=\sum_{j=1}^{15}e^{\frac{2\pi ij}{15}}R_{n, j},\;Z_{n}^{\prime}=\sum
_{j=1}^{21}e^{\frac{2\pi ij}{21}}S_{n, j}.
\]

For each $n$, let $\Omega_{n}=\sum_{k=1}^{s}\lambda_{k}P_{n,k}$. Then $\Omega=\left \{
\Omega_{n}\right \}  _{\alpha}.$ For each $n\in \mathbb{N}$ there is a unital $\ast
$-homomorphism $\pi_{n}: C^*(G) \rightarrow \mathcal{A}_{n}$ such that%
\[
\pi_{n}\left(  x\right)  =X_{n}^{\prime}\Omega_{n}^{1/10}, \; \pi_{n}\left(
y\right)  =Y_{n}^{\prime}\Omega_{n}^{1/15}, \; \pi_{n}\left(  z\right)
=Z_{n}^{\prime}\Omega_{n}^{1/21}.
\] (Here again by  $\Omega_n^{1/10}$, etc., we mean the normal operator obtained
by applying the Borel function
$z^{1/10} =_{def} |z|^{1/10}e^{\frac{iArg z}{10}}$, etc.,  to $\Omega_n$. Since $\Omega_n$ has finite spectrum, $\Omega_n^{1/10}$, etc.,
belong to $\mathcal{A}_{n}$.)
Clearly, $$\left \{  \pi_{n}\left(  x\right)  \right \}  _{\alpha}=X^{\prime}\Omega^{
1/10}, \left \{  \pi_{n}\left(  y\right)  \right \}  _{\alpha}=Y^{\prime
}\Omega^{1/15}, \left \{  \pi_{n}\left(  z\right)  \right \}  _{\alpha}=Z^{\prime
}\Omega^{1/21}.$$ By (\ref{estimate1}), (\ref{estimate2}), (\ref{estimate3}) and
Lemma \ref{LimitOfLiftable}, $G $ is $RR0$-stable.
\end{proof}

\subsection{Groups of the form (\ref{1-relatorgroupCase2}).}

We will need a few easy lemmas. The first lemma is folklore.

\begin{lemma}
\label{Lemma2FromSeparateFile} Let $\mathcal M$ be a $II_1$-factor, $p\in \mathcal{M}$ be a projection and
$0\le \beta \le \tau(p)$. Then there is a projection $p^{\prime}\in p\mathcal{M
}p$ such that $\tau(p^{\prime}) = \beta$.
\end{lemma}
\begin{proof} It follows from folklore fact that in $II_1$-factor one can find a projection
with prescribed trace.
\end{proof}

\begin{lemma}
\label{LiftProjectionPreservingTrace} Let $p\in \prod_{\alpha} (\mathcal{M}%
_{i}, \rho_{i})$ be a projection. Then $p$ can be lifted to a projection
$\{P_{i}\} \in \prod \mathcal{M}_{i}$ with $\rho_{i}(P_{i}) = \rho(p)$, for all
$i$.
\end{lemma}
\begin{proof}
Lift $p$ to a projection $\{ \tilde P_{i}\}$. Then
\[
\rho_{i}(\tilde P_{i}) - \rho(p) \to_{\alpha} 0.
\]
If $\rho_{i}(\tilde P_{i}) \ge \rho(p)$, then by Lemma
\ref{Lemma2FromSeparateFile} there is a projection $Q_{i}\in \tilde
P_{i}\mathcal{M}_{i}\tilde P_{i}$ such that $\rho_{i}(Q_{i}) = \rho_{i}(\tilde
P_{i}) - \rho(p).$ Let
\[
P_{i} = \tilde P_{i} - Q_{i}.
\]
If $\rho_{i}(\tilde P_{i}) \le \rho(p)$, then $\rho_{i}(1-\tilde P_{i}) \ge
\rho(1-p)$. By Lemma \ref{Lemma2FromSeparateFile} there is a projection
$Q_{i}\in(1-\tilde P_{i})\mathcal{M}_{i} (1-\tilde P_{i})$ such that $\rho
_{i}(Q_{i}) = \rho_{i}(1-\tilde P_{i}) - \rho(p).$ In this case let
\[
P_{i} = \tilde P_{i} + Q_{i}.
\]
Either way $P_{i}$ is a projection and $\rho_{i}(P_{i}) = \rho(p).$  We have
\[
\rho_{i}(P_{i} - \tilde P_{i}) = \rho(p) - \rho_{i}(\tilde P_{i}) \to_{\alpha}
0
\]
and hence $\{P_{i}\}$ is a lift of $p$.
\end{proof}

\begin{lemma}
\label{LiftLinearEquationPreservinTrace} Let $p, q_{1}, \ldots, q_{n}\in
\prod_{\alpha} (\mathcal{M}_{i}, \rho_{i})$ be projections and $\sum_{k=1}^{n}
q_{k} = p$. Suppose $p$ is lifted to a projection $\{P_{i}\}$ with $\rho
_{i}(P_{i}) = \rho(p)$. Then each $q_{k}$ can be lifted to a projection
$\{Q_{k, i}\}$ such that for all $i$
\[
\sum_{k=1}^{n} Q_{k, i} = P_{i}%
\]
and for all $i, k$%
\[
\rho_{i}(Q_{k, i}) = \rho(q_{k}).
\]

\end{lemma}
\begin{proof}
By Lemma \ref{LiftProjectionPreservingTrace} we can lift $q_1$ to $\{Q_{1, i}\} \in \prod_{\alpha} (P_{i}\mathcal{M}_{i}%
P_{i}, \frac{\rho_{i}}{\rho_{i}(P_{i})})$ with $\rho_{i}(Q_{1, i}) =
\rho(q_{1}).$ Now, again by  Lemma \ref{LiftProjectionPreservingTrace}, we can lift $q_2$ to $$\{Q_{2, i}\} \in \prod_{\alpha} \left(\left(P_{i} - Q_{1, i}\right)\mathcal{M}_{i}%
\left(P_{i}- Q_{1, i}\right), \frac{\rho_{i}}{\rho_{i}\left(P_{i} - Q_{1, i}\right)}\right)$$ with $\rho_{i}(Q_{2, i}) =
\rho(q_{2}).$ Then we lift $q_3$ to $$\{Q_{3, i}\} \in \prod_{\alpha} \left(\left(P_{i} - \sum_{k=1}^2 Q_{k, i}\right)\mathcal{M}_{i}%
\left(P_{i}- \sum_{k=1}^2 Q_{k, i}\right), \frac{\rho_{i}}{\rho_{i}\left(P_{i} - \sum_{k=1}^2 Q_{k, i}\right)}\right)$$ with $\rho_{i}(Q_{3, i}) =
\rho(q_{3}).$ Continuing this process we obtain pairwisely orthogonal lifts $\{Q_{k, i}\} \in \prod_{\alpha} (P_{i}\mathcal{M}_{i}%
P_{i}, \frac{\rho_{i}}{\rho_{i}(P_{i})})$ of $q_k$, $k \le n-1$, such that  $\rho_{i}(Q_{k, i}) =
\rho(q_{k}).$
 Now we lift $q_{n}$  to the projection $\{P_{i} - \sum
_{k=1}^{n-1} Q_{k, i}\}\in \prod_{\alpha} (P_{i}\mathcal{M}_{i}%
P_{i}, \frac{\rho_{i}}{\rho_{i}(P_{i})})$. Then for all i
\[
\rho_{i}\left( P_{i} - \sum_{k=1}^{n-1}Q_{k, i}\right)  = \rho(p) - \rho \left(
\sum_{k=1}^{n-1}q_{k}\right)  = \rho(q_{n})
\]
and
\[
\sum_{k=1}^{n} Q_{k, i} = P_{i}%
.\]

\end{proof}

\begin{theorem}
\label{1-relator2} Let $G$ be as in (\ref{1-relatorgroupCase2}). Then
G is $II_{1}$-factor stable.
\end{theorem}
We would like to warn the reader that the notation in the proof below differs slightly from the notation
in the proof of Theorem \ref{1-relator1}: $\tilde \Omega$ now plays the role of $W$  and $Q_i$'s play the role of $P_i$'s.
\begin{proof}
Suppose $\left \{  \left(  \mathcal{A}_{n},\rho_{n}\right)  \right \}  $ is a
sequence of $II_{1}$-factors, $\alpha$ is a non-trivial ultrafilter on
$\mathbb{N}$, and $u, x_{1}, \ldots, x_{m}\in%
{\displaystyle \prod_{n\in \mathbb{N}}^{\alpha}}
\left(  \mathcal{A}_{n},\rho_{n}\right)  =_{\text{\textrm{def}}}\left(
\mathcal{A},\rho \right)  $ are unitaries satisfying the group relations. Then
\[
x_{1}^{a_{1}\ldots a_{m-1}} = x_{2}^{b_{1}a_{2}\ldots a_{m-1}} = \ldots=
x_{m}^{b_{1}\ldots b_{m-1}}=_{\text{\textrm{def}}}\tilde \Omega.
\]
Obviously $x_{1}, \ldots, x_{m}$ commute with $\tilde \Omega $. Let
\[
N_{i} = b_{1}\ldots b_{i-1}a_{i}\ldots a_{m-1}.
\]
Since $N_{1} = N_{m} $,
\[
u\tilde \Omega u^{-1} = (ux_{1}u^{-1})^{N_{1}} = x_{m}^{N_{1}} = x_{m}^{N_{m}} = \tilde \Omega.
\]
Thus $u, x_{1}, \ldots, x_{m}$ commute with $\tilde\Omega$ and hence with every spectral
projection of $\tilde \Omega$, and since $\mathcal{A}$ is a von Neumann algebra, the
spectral projections of $\tilde \Omega$ are in $\mathcal{A}$.

Let $\varepsilon>0$. We can choose an orthogonal family of nonzero spectral
projections $\left \{  Q_{1},\ldots, Q_{s}\right \}  $ of $\tilde \Omega$ whose sum is $1$
and we can choose $\lambda_{1},\ldots,\lambda_{s}\in \mathbb{T}$ such that if
$\Omega=\sum_{k=1}^{s}\lambda_{k}Q_{k}$, then%
\[
\left \Vert \tilde \Omega ^{i}-\Omega^{i}\right \Vert _{2}<\varepsilon,
\]
for $i\in \left \{  1,\frac{1}{N_{1}}, \ldots, \frac{1}{N_{m}}\right \}  .$
Here and below  by  $\Omega^{t}$, $\tilde \Omega^t$ etc., we mean the normal operator obtained
by applying the Borel function
$z^{t} =_{def} |z|^{t}e^{itArg z}$, etc.,  to $\Omega, \tilde \Omega$, etc.
Clearly
\[
T=\sum_{k=1}^{s}Q_{k}TQ_{k}%
\]
for $T\in \left \{  u, x_{1}, \ldots, x_{m}, \tilde \Omega, \Omega\right \}  .$ For each $n$ we
can find an orthogonal family $\left \{  Q_{n,1},\ldots, Q_{n,s}\right \}  $ of
projections in $\mathcal{A}_{n}$ whose sum is $1$ such that, for $1\leq k\leq
s,$%
\[
Q_{k}=\left \{  Q_{n,k}\right \}  _{\alpha}\text{ .}%
\]
It is clear that $\sum_{k=1}^{s}Q_{k}\mathcal{A}Q_{k}$ is the tracial
ultraproduct $%
{\displaystyle \prod_{n\in \mathbb{N}}^{\alpha}}
\left(  \sum_{k=1}^{s}Q_{n,k}\mathcal{A}_{n}Q_{n,k},\rho_{n}\right)  $ and
that each $\left(  Q_{k}\mathcal{A}Q_{k},\frac{1}{\rho \left(  Q_{k}\right)
}\rho \right)  $ is the tracial ultraproduct $%
{\displaystyle \prod_{n\in \mathbb{N}}^{\alpha}}
\left(  Q_{n,k}\mathcal{A}_{n}Q_{n,k},\frac{1}{\rho_{n}\left(  Q_{n,k}\right)
}\rho_{n}\right)  $. Let
\[
x_{i}^{\prime} = x_{i}\tilde \Omega^{-\frac{1}{N_{i}}},
\]
$i = 1, \ldots, m.$ Then
\begin{equation}
\label{estimate11}\| x_{i}-x_{i}^{\prime}\Omega^{1/{N_{i}}}\|_{2}<\varepsilon,
\end{equation}
and
\begin{equation}
\label{newequation}ux_{1}^{\prime}u^{-1} = x_{m}^{\prime}%
\end{equation}
\begin{equation}
\label{old1}(x_{1}^{\prime})^{a_{1}} = (x_{2}^{\prime})^{b_{1}},
(x_{2}^{\prime})^{a_{2}} = (x_{3}^{\prime})^{b_{2}}, \ldots, (x_{m-1}^{\prime
})^{a_{m-1}} = (x_{m}^{\prime})^{b_{m-1}}%
\end{equation}%
\begin{equation}
\label{old2}(x_{1}^{\prime})^{a_{1}\ldots a_{m-1}} = \ldots= (x_{m}^{\prime
})^{b_{1}\ldots b_{m-1}} = 1.
\end{equation}

We notice also that $u, x_{1}^{\prime}\Omega^{1/{a_{1}\ldots a_{m-1}}}, \ldots,
x_{m}^{\prime}\Omega^{1/{b_{1}\ldots b_{m-1}}}$ satisfy the group relations.

Now we are going to "lift" the relations (\ref{newequation}) -- (\ref{old2})
and we will do it in two steps.

\bigskip

\textbf{STEP 1}: To "lift" the relations (\ref{old1}) and (\ref{old2}) so that
$x_{1}^{\prime}$ and $x_{m}^{\prime}$ will be lifted to $\{X^{\prime(1)}_{
n}\}$, $\{X^{\prime(m)}_{n}\} \in \prod \left(  \sum_{k=1}^{s}Q_{n,k}%
\mathcal{A}_{n}Q_{n,k}\right) $ unitarily equivalent to each other. (Possibly
this unitary equivalence won't be a lift of $u$.)

\bigskip To do STEP 1 we notice that the relation (\ref{old2}) implies that
each $x_{i}^{\prime}$ can be written as a linear combination of projections
\[
x_{i}^{\prime}= \sum_{k=1}^{N_{i}} e^{\frac{2\pi i k}{N_{i}} }p^{(i)}_{k},
\]
$i = 1, \ldots, m$. In (\ref{old1}) each relation
\[
(x_{i}^{\prime})^{a_{i}} = (x_{i+1}^{\prime})^{b_{i}}%
\]
now translates into a system of linear equations with some of $p_{k}^{(i)}$,
$k= 1, \ldots, N_{i}$, in the left-hand sides and some of $p_{k}^{(i+1)}$, $k=
1, \ldots, N_{i+1}$, in the right-hand sides. (We don't write out the details
since we did it in the proof of Theorem \ref{1-relator1}).

Since each $p_{k}^{(i)}$ is the direct sum of the projections $Q_{j}%
p_{k}^{(i)}Q_{j}$, this system of linear equations translates into $s$ systems
of linear equations, one for each coordinate. Thus for each $j = 1, \ldots, s$
we have a system of linear equations with some of $Q_{j}p_{k}^{(i)}Q_{j}$, $k=
1, \ldots, N_{i}$, in the left-hand sides and some of $Q_{j}p_{k}^{(i+1)}%
Q_{j}$, $k= 1, \ldots, N_{i+1}$, in the right-hand sides.

We notice also that (\ref{newequation}) implies that
\[
up_{k}^{(1)}u^{-1} = p_{k}^{(m)}%
\]
for all $k= 1, \ldots, N = N_{1} = N_{m}$. In particular
\begin{equation}
\label{TracesEqual}\rho \left( Q_{j}p_{k}^{(1)}Q_{j}\right)  = \rho \left(
Q_{j}p_{k}^{(m)}Q_{j}\right)
\end{equation}
for all $k = 1, \ldots, N$, $j=1, \ldots, s$.

By Lemma \ref{LiftProjectionPreservingTrace} we can lift each projection
$Q_{j}p_{k}^{(1)}Q_{j}$ to a projection $\{P_{k, n, j}^{(1)}\} \in \prod \left(
Q_{n,j}\mathcal{A}_{n}Q_{n,j}\right) $ of the same trace as $Q_{j}p_{k}%
^{(1)}Q_{j}$. By Lemma \ref{LiftLinearEquationPreservinTrace} we can lift each
$Q_{j}p_{k}^{(2)}Q_{j}$ to a projection $\{P_{k, n, j}^{(2)}\} \in \prod \left(
Q_{n,k}\mathcal{A}_{n}Q_{n,k}\right) $ of the same trace as $Q_{j}p_{k}%
^{(2)}Q_{j}$ and such that the family

\noindent$\left \{ \{P_{k, n, j}^{(1)}\}, \{P_{l, n, j}^{(2)}\} \;|\; k=1,
\ldots, N, \;l = 1, \ldots, N_{2}\right \}  $ would satisfy the same linear
relations as the family $\{Q_{j}p_{k}^{(1)}Q_{j}, Q_{j}p_{l}^{(2)}Q_{j} \;|\;
k=1, \ldots, N, \;l = 1, \ldots, N_{2}\}$. We keep doing this. We end up with
projections $\{P_{k, n, j}^{(i)}\}$, $i= 1, \ldots, m$, $k=1, \ldots, N_{i}$,
$j = 1, \ldots, s$ of the same trace as $Q_{j}p_{k}^{(i)}Q_{j}$ and satisfying
the same system of linear relations. In particular, by (\ref{TracesEqual}) we
have
\[
\rho_{n}(P_{k, n, j}^{(1)}) = \rho_{n}(P_{k, n, j}^{(m)}),
\]
$n\in \mathbb{N}$, $k= 1, \ldots, N$, $j = 1, \ldots, s$. Then there is unitary $\{W_{n, j}\} \in
\prod \left( Q_{n,k}\mathcal{A}_{n}Q_{n,k}\right) $ such that
\begin{equation}
W_{n, j} P_{k, n, j}^{(1)} W_{n, j}^{-1} = P_{k, n, j}^{(m)}%
\end{equation}
for all $n, k, j$. Let for each $k, n, i$
\[
P_{k, n}^{(i)} = \sum_{j=1}^{s} P_{k, n, j}^{(i)}.
\]
For each $n$ let
\[
W_{n} = \sum_{j=1}^{s} W_{n, j}.
\]
Then the projections $\{P_{k, n}^{(i)}\}$, $k= 1, \ldots, N_{i}$, $i = 1,
\ldots, m$ are lifts of $p_{k}^{(i)}$ and satisfy the same system of linear
equations. We have also
\begin{equation}
\label{liftsx_1andx_munitarilyequivalent}W_{n} P_{k, n}^{(1)} W_{n}^{-1} =
P_{k, n}^{(m)}%
\end{equation}
for all $n, k$. Let
\[
X^{\prime(i)}_{ n} = \sum_{k=1}^{N_{i}} e^{\frac{2\pi i k}{N_{i}} }P^{(i)}_{k,
n},
\]
$i= 1, \ldots, m$, $k=1, \ldots, N_{i}$, $n\in \mathbb{N}$. Then the unitaries
$\{X^{\prime(i)}_{ n}\}$, $i= 1, \ldots, m$, are lifts of $x_{i}$'s and
satisfy the relations (\ref{old1}) and (\ref{old2}). It follows from
(\ref{liftsx_1andx_munitarilyequivalent}) that $\{X^{\prime(1)}_{ n}\}$,
$\{X^{\prime(m)}_{n}\} \in \prod \left(  \sum_{k=1}^{s}Q_{n,k}\mathcal{A}%
_{n}Q_{n,k}\right) $ are unitarily equivalent to each other. STEP 1 is done.

\bigskip

\bigskip

\textbf{STEP 2}: Given the lifts $\{X^{\prime(i)}_{n}\}$ of $x^{\prime}_{i}$,
$i=1, \ldots, m$, constructed in STEP 1, to find a lift of $u$ which would
conjugate $\{X_{n}^{\prime(1)}\}$ and $\{X_{n}^{\prime(m)}\}$.

\bigskip

At first we lift $u$ to anything, say $\{X_{n}\} \in \prod \left(  \sum_{k=1}%
^{s}Q_{n,k}\mathcal{A}_{n}Q_{n,k}\right) $, that is
\[
\{X_{n}\}_{\alpha} = u.
\]
Let for each n
\[
\tilde X_{n} = \sum_{k=1}^{N} P_{k, n}^{(m)}X_{n}P_{k, n}^{(1)},
\]
where $P_{k, n}^{(m)}$ and $P_{k, n}^{(1)}$ are projections constructed in
STEP 1. Then
\[
P_{k, n}^{(m)}\tilde X_{n} = \tilde X_{n} P_{k, n}^{(1)}%
\]
for all $k, n$ and
\[
\{ \tilde X_{n}\}_{\alpha} = u,
\]
because
\[
\sum p_{k}^{(m)}up_{k}^{(1)} = \sum up_{k}^{(1)}p_{k}^{(1)} = \sum
up_{k}^{(1)} = u.
\]
We are going to show that the unitary from the polar decomposition of $\tilde
X_{n}$ also will conjugate $P_{k, n}^{(m)}$ and $P_{k, n}^{(1)}$, for all $k$'s.

By (\ref{liftsx_1andx_munitarilyequivalent}), for each $n$  we have
\begin{equation}
\label{2.2FromSeparateFile}\tilde X_{n} = W_{n}\sum_{k=1}^{N} P_{k, n}%
^{(1)}W_{n}^{-1}X_{n} P_{k, n}^{(1)}.
\end{equation}
For each $k, n, j$, $\; \;P_{k, n, j}^{(1)}\mathcal{A}_{n} P_{k, n, j}^{(1)}$
is a $II_{1}$-factor. As is well known, in $II_{1}$-factors a partial isometry
in polar decomposition can always be chosen unitary. Since
\[
P_{k, n}^{(1)}\left(  \sum_{j=1}^{s}Q_{n,j}\mathcal{A}_{n}Q_{n,j}\right)
P_{k, n}^{(1)} = \oplus_{j=1}^{s} P_{k, n, j}^{(1)}\mathcal{A}_{n} P_{k, n,
j}^{(1)},
\]
for each $k, n$ we have
\[
P_{k, n}^{(1)}W_{n}^{-1}X_{n}P_{k,n}^{(1)} = V_{k, n}\left| P_{k, n}%
^{(1)}W_{n}^{-1}X_{n}P_{k,n}^{(1)}\right|
\]
with $V_{k, n}\in P_{k, n}^{(1)}\left(  \sum_{j=1}^{s}Q_{n,j}\mathcal{A}%
_{n}Q_{n,j}\right)  P_{k, n}^{(1)}$ being unitary.  Let
\[
V_{n} = \sum_{k=1}^{N} V_{k, n}.
\]
It is unitary and
\begin{equation}
\label{2.3FromSeparateFile}V_{n} = \sum_{k=1}^{N} P_{k, n}^{(1)} V_{n} P_{k,
n}^{(1)}.
\end{equation}
We have, by (\ref{liftsx_1andx_munitarilyequivalent}) and
(\ref{2.2FromSeparateFile}),
\begin{multline}
\tilde X_{n} = W_{n}V_{n}\left| \sum_{k=1}^{N} P_{k, n}^{(1)}W_{n}^{-1}%
X_{n}P_{k, n}^{(1)}\right|  = W_{n}V_{n}\left| \sum_{k=1}^{N} W_{n}^{-1}P_{k,
n}^{(m)}W_{n}W_{n}^{-1}X_{n}P_{k, n}^{(1)}\right|  =\\
W_{n}V_{n}\left| W_{n}^{-1}\sum_{k=1}^{N} P_{k, n}^{(m)}X_{n}P_{k, n}%
^{(1)}\right|  = W_{n}V_{n}\left| \sum_{k=1}^{N} P_{k, n}^{(m)}X_{n}P_{k,
n}^{(1)}\right|  = W_{n}V_{n}\left| \tilde X_{n}\right| .
\end{multline}
Let
\[
\tilde V_{n} = W_{n} V_{n}.
\]
Then
\[
\tilde X_{n} = \tilde V_{n} \left| \tilde X_{n}\right|
\]
and since $\{ \tilde X_{n}\}_{\alpha} = u$, we conclude that
\[
\{ \tilde V_{n}\}_{\alpha} = u.
\]
By (\ref{2.3FromSeparateFile})
\[
\tilde V_{n} = W_{n}\sum_{k=1}^{N} P_{k, n}^{(1)}V_{n}P_{k, n}^{(1)} =
\sum_{k=1}^{N} P_{k, n}^{(m)}W_{n}V_{n}P_{k, n}^{(1)}.
\]
Hence
\[
P_{k, n}^{(m)}\tilde V_{n} = \tilde V_{n} P_{k, n}^{(1)}%
\]
which implies that
\[
X_{n}^{\prime(m)}\tilde V_{n} = \tilde V_{n} X_{n}^{\prime(1)}.
\]
STEP 2 is done.

\bigskip

For each $n$, let $\Omega_{n}=\sum_{k=1}^{s}\lambda_{k}Q_{n,k}$. Then $\Omega=\left \{
\Omega_{n}\right \}  _{\alpha}.$ Let
\[
X_{n}^{(i)} = X_{n}^{\prime(i)}\Omega_{n}^{1/N_{i}}.
\]
Then by (\ref{estimate11})
\[
\|x_{i} - \{X_{n}^{(i)}\}_{\alpha}\|_{2} \le \epsilon.
\]
Since $\tilde V_{n}, X_{n}^{(1)}, \ldots, X_{n}^{(m)}$ satisfy the group
relations, Lemma \ref{LimitOfLiftable} completes the proof.
\end{proof}

Theorems \ref{1-relator1} and \ref{1-relator2} imply that

\begin{theorem}
\label{1-relatorTheorem} One-relator groups with a nontrivial center are
$II_{1}$-factor stable.
\end{theorem}


\subsection{RFD}

Recall that a $C^*$-algebra is {\it residually finite-dimensional} (RFD) if it has a separating family of finite-dimensional representations.

Though property of being RFD is not directly related to stability,
arguments similar to ones used in the proof of Theorem \ref{1-relator1} can be
applied to show that for groups of the form (\ref{1-relatorgroupCase1}),
$C^{*}(G)$ is RFD. We will need a lemma.

\begin{lemma}
\label{projectionsRFD} Suppose $Q, R_{1}, \ldots, R_{N}\in B(H)$ are
projections and $\sum_{i=1}^{N} R_{i} = Q$, $P_{n}\in B(H)$ are finite-rank
projections and $P_{n}\uparrow1$, $Q_{n}\in P_{n}B(H)P_{n}$ and SOT-$\lim
Q_{n} = Q$. Then there exists projections $R_{n}^{(i)}\in P_{n}B(H)P_{n}$ such
that SOT-$\lim R_{n}^{(i)} = R_{i}$ and $\sum_{i=1}^{N} R_{n}^{(i)} = Q_{n}.$
\end{lemma}

\begin{proof}
Let $\tilde H = Q(H).$ Then $Q$ is the unit in $B(\tilde H)$. Since
projections with sum 1 generate a commutative $C^{*}$-algebra, hence RFD, the
statement follows from [\cite{DonRFD}, Th. 11].
\end{proof}

\begin{theorem}\label{RFD}
Let $G$ be of the form (\ref{1-relatorgroupCase1}). Then $C^{*}(G)$ is RFD.
\end{theorem}

\begin{proof}
Again we will do it for the case $G= \left<  x, y, z\;|\; x^{2} = y^{3}, y^{5}
= z^{7}\right>  $, and the proof for the general case is analogous. Let $0\neq
a\in C^{*}(G).$ Then there exists an irreducible representation $\pi$ of
$C^{*}(G)$ such that $\pi(a)\neq0$. The representation $\pi$ must factorize
through the $C^{*}$-algebra
\[
C^{*} \left(  x, y, z \;| \; x^{2} = y^{3}, y^{5} = z^{7}, \; x^{10}=y^{15} =
z^{21} = 1 \right)  .
\]
Indeed, $\pi \left(  x^{10}\right)  =\pi \left(  y^{15}\right)  = \pi \left(
z^{21}\right)  \in \pi \left(  C^{\ast}\left(  G\right)  \right)  ^{\prime
}=\mathbb{C}1$. Hence there is $\lambda \in \mathbb{T}$ such that $\pi(x^{10}) =
\pi(y^{15}) = \pi(z^{21})= \lambda$. Then there is an isomorphism
\[
\pi(C^{*}(G))\cong C^{*} \left(  x, y, z \;| \; x^{2} = y^{3}, y^{5} = z^{7},
\; x^{10}=y^{15} = z^{21} = 1 \right)
\]
given by
\[
x^{\prime} \mapsto{\pi(x)}{\lambda}^{-1/10}, \; y^{\prime} \mapsto
\pi(y){\lambda}^{-1/15}, \;z^{\prime} \mapsto \pi(z){\lambda}^{-1/21}.
\]
Here   by ${\lambda}^{-1/10}$ etc.  we mean
$ |\lambda|^{-1/10}e^{-\frac{iArg \lambda}{10}}$ etc.
By arguments used in the proof of Theorem \ref{1-relator1}, the latter algebra
is isomorphic to the universal $C^{*}$-algebra $\mathcal{D}$ of the relations
\begin{multline}
q_{1}+q_{6} = r_{1}+r_{6}+r_{11}\\
q_{2}+q_{7} = r_{2}+r_{7}+r_{12}\\
q_{3}+q_{8} = r_{3}+r_{8}+r_{13}\\
q_{4}+q_{9} = r_{4}+r_{9}+r_{14}\\
q_{5}+q_{10} = r_{5}+r_{10}+r_{15}\\
r_{1}+r_{4}+r_{7}+r_{10}+r_{13} = s_{1}+s_{4}+s_{7}+s_{10}+s_{13}%
+s_{16}+s_{19}\\
r_{2}+r_{5}+r_{8}+r_{11}+r_{14} = s_{2}+s_{5}+s_{8}+s_{11}+s_{14}%
+s_{17}+s_{20}\\
r_{3}+r_{6}+r_{9}+r_{12}+r_{15} = s_{3}+s_{6}+s_{9}+s_{12}+s_{15}%
+s_{18}+s_{21},
\end{multline}
where all $q_{i}, r_{k}, s_{m}, i=1,\ldots, 10, k=1, \ldots, 15, m = 1,
\ldots, 21 $, are projections.

Thus $\pi= \psi \circ j$, where $j: C^{*}(G)\to \mathcal{D }$ and $\psi$ is a
representation of $\mathcal{D }$. It follows from Lemma \ref{projectionsRFD},
that any representation of $\mathcal{D}$ is a pointwise SOT-limit of
finite-dimensional representations. Hence there exists a finite-dimensional
representation $\phi$ of $\mathcal{D }$ such that $\phi(j(a)) \neq0$. Thus
there is a finite-dimensional representation (namely $\phi \circ j$) of
$C^{*}(G)$ that does not vanish on $a$.
\end{proof}

\end{document}